\documentclass{article}
\usepackage[top=1in,bottom=1.25in,left=1.25in,right=1.25in]{geometry}
\usepackage{amsmath}
\usepackage{amsthm}
\usepackage{appendix}
\usepackage{amssymb}
\usepackage{graphicx}
\usepackage{epsfig}
\usepackage{color}
\usepackage{xcolor}
\usepackage{newfloat}
\usepackage{algorithm}
\usepackage{algorithmic}
\usepackage{caption}
\usepackage{booktabs}
\usepackage{multirow}
\usepackage{mathtools}
\usepackage{enumerate} 
\usepackage{comment} 

\newtheorem{theo}{Theorem}[section] 
\newtheorem{defi}[theo]{Definition}

\newtheorem{lemm}[theo]{Lemma}
\newtheorem{rema}[theo]{Remark}
\newtheorem{coro}[theo]{Corollary}

\newtheorem{example}[theo]{Example}

\newcommand{\beq}{\begin{equation}}
\newcommand{\eeq}{\end{equation}}
\newcommand{\beqs}{\begin{equation*}}
\newcommand{\eeqs}{\end{equation*}}

\newcommand{\D}{\mathbb D}
\newcommand{\C}{\mathbb C}
\newcommand{\R}{\mathbb R}
\newcommand{\K}{\mathbb K}

\newcommand{\imagunit}{{\bf i}}

\newcommand{\conjg}[1]{\overline{#1}} 

\newcommand{\goes}{\rightarrow}

\newcommand{\rank}[1]{\mathrm{rank}#1} 
\newcommand{\diag}{\mathrm{diag}} 
\newcommand{\Cmn}[2]{\C^{{#1}\times{#2}}}  
\newcommand{\Rmn}[2]{\R^{{#1}\times{#2}}}

\newcommand{\pertmat}{\Delta}

\DeclareMathOperator{\Tr}{Tr}

\renewcommand\Re{\operatorname{Re}}
\renewcommand\Im{\operatorname{Im}}
\newcommand\real[1]{\Re{(#1)}}

\newcommand{\Real}[1]{\Re{\left({#1}\right)}}

\DeclareFloatingEnvironment[placement=!ht]{algfloat}
\newcommand{\algnote}[1]{\footnotesize \sc{Note: \it#1 } }

\renewcommand{\algorithmicensure}{\textbf{Output:}}

\newcommand{\algorithmicpurpose}[1]{
	\renewcommand{\algorithmicensure}{\textbf{Purpose:}}
	\ENSURE{#1}
	\renewcommand{\algorithmicensure}{\textbf{Output:}}
}

\def\bigo{{\mathcal O}}
\def\iu{\imagunit}

\def\eps{\varepsilon}

\def\iu{\imagunit}

\def\eps{\varepsilon}
\def\bigo{{\mathcal O}}

\newcommand{\aleps}{\alpha_\varepsilon}

\newcommand{\rhoeps}{\rho_\varepsilon}

\newcommand{\alepsKn}{\aleps^{\K,\norm{\cdot}}}
\newcommand{\alepsRn}{\aleps^{\R,\norm{\cdot}}}
\newcommand{\alepsRf}{\aleps^{\R,\normf{\cdot}}}

\newcommand{\rhoepsKn}{\rhoeps^{\K,\norm{\cdot}}}

\newcommand{\rhoepsRf}{\rhoeps^{\R,\normf{\cdot}}}

\newcommand{\Hinf}{H_\infty}
\newcommand{\eigtfmat}{A+B\Delta(I-D \Delta)^{-1}C} 

\newcommand{\SVSReps}{\spec^\R_\eps}

\newcommand{\SVSRepsf}{\spec^{\R,\|\cdot\|\fro}_\eps}
\newcommand{\SVSRepss}{\spec^{\R,\|\cdot\|_2}_\eps}

\newcommand{\SVSCepsn}{\spec^{\C,\|\cdot\|}_\eps}
\newcommand{\SVSCepsf}{\spec^{\C,\|\cdot\|\fro}_\eps}
\newcommand{\SVSCepss}{\spec^{\C,\|\cdot\|_2}_\eps}
\newcommand{\SVSKepss}{\spec^{\K,\|\cdot\|_2}_\eps}
\newcommand{\SVSKepsn}{\spec^{\K,\|\cdot\|}_\eps}
\newcommand{\stabradKn}{r_{\mathbb K}^{\norm{\cdot}}} 
\newcommand{\SVSKepsf}{\spec^{\K,\|\cdot\|\fro}_\eps}
\newcommand{\stabradRf}{r_{\mathbb R}^{\normf{\cdot}}} 
\newcommand{\stabradCn}{r_{\mathbb C}^{\norm{\cdot}}} 
\newcommand{\stabradRs}{r_{\mathbb R}^{\norms{\cdot}}} 

\newcommand{\SVSKzeron}{\spec_0^{\K,\|\cdot\|}}

\newcommand{\epslb}{\eps_{\mathrm{LB}}}
\newcommand{\epsub}{\eps_{\mathrm{UB}}}

\newcommand{\mur}{\mu_{\R}^{\norm{\cdot}}}
\newcommand{\murs}{\mu_{\R}^{\norms{\cdot}}}
\newcommand{\murf}{\mu_{\R}^{\normf{\cdot}}}
\newcommand{\muk}{\mu_{\K}^{\norm{\cdot}}}

\newcommand{\spec}{\sigma} 

\newcommand{\zlam}{\lambda} 

\newcommand{\norm}[1]{\left\lVert#1\right\rVert}
\newcommand{\normf}[1]{\norm{#1}\fro}
\newcommand{\norms}[1]{\norm{#1}_2}

\newcommand{\transsym}{\mathsf{T}}
\newcommand{\tp}[1][]{^{{#1}\transsym}}
\newcommand{\frosym}{\mathrm{F}}
\newcommand{\fro}{_\frosym}

\newcommand{\Reuv}{\Re\left(uv^*\right)}
\newcommand{\Reuvt}{\Re\left(u(t)v(t)^*\right)}

\newcommand{\Reuvo}{\Re\left(u_0 v_0^*\right)}
\newcommand{\Reuvk}{\Re\left(u_k v_k^*\right)}

\newcommand{\olhp}{\C_-} 
\newcommand{\cuhp}{H_{+=}} 
\newcommand{\oud}{\D_-} 

\newcommand{\epsstar}{\eps_\star}

\newcommand{\granso}{{\sc granso}}

\newcommand{\mytilde}{\raise.17ex\hbox{$\scriptstyle\mathtt{\sim}$}}

\newcommand{\matlab}{{\sc Matlab}}

\title{Approximating the Real Structured Stability Radius with Frobenius Norm Bounded Perturbations}
\author{N.~Guglielmi\thanks{Dipartimento di Matematica
Pura e Applicata, Universit\`a dell'Aquila, via Vetoio (Coppito),
I-$67010$ L'Aquila, Italy (guglielm@univaq.it). Supported in part
by the Italian Ministry of Education, Universities and Research
(M.I.U.R.) and by Istituto Nazionale di Alta Matematica - Gruppo Nazionale
 di Calcolo Scientifico (INdAM-G.N.C.S).\ }
\and M.~G\"urb\"uzbalaban\thanks{Department of Management Science and Information Systems, Rutgers University, New Brunswick, NJ, USA (mgurbuzbalaban@business.rutgers.edu)}
\and T.~Mitchell\thanks{
Max Planck Institute for Dynamics of Complex Technical Systems, Magdeburg, 39106 Germany (mitchell@mpi-magdeburg.mpg.de). Previously supported by
National Science Foundation Grant DMS-1317205 at New York University.}
\and M.~L.~Overton\thanks{Courant Institute of
Mathematical Sciences, New York University, 251 Mercer Street, New
York, NY 10012, USA (overton@cims.nyu.edu). Supported in part by
National Science Foundation Grant DMS-1620083.}
}
\date{Dec. 30th, 2016}

\begin{document} 
\maketitle

\begin{abstract}
We propose a fast method to approximate the real stability radius of a linear dynamical system
with output feedback, where the perturbations are restricted to be real valued and bounded with respect to the Frobenius norm.  Our work builds on a number of scalable algorithms that have been proposed in recent years, ranging from methods that approximate the complex or real pseudospectral abscissa and radius of large sparse matrices (and generalizations of these methods for pseudospectra to spectral value sets) to algorithms for approximating the complex stability radius (the reciprocal of the $H_\infty$ norm).  Although our algorithm is guaranteed to find only upper bounds to the real stability radius, it seems quite effective in practice.  As far as we know, this is the first algorithm that addresses the Frobenius-norm version of this problem.  
Because the cost mainly consists of computing the eigenvalue with maximal real part for continuous-time systems (or modulus for discrete-time systems) of a sequence of matrices, our algorithm remains very efficient for large-scale systems provided that the system matrices are sparse.

\end{abstract}

\section{Introduction}
\label{sec:intro}

Consider a linear time-invariant dynamical system with output feedback defined, for continuous-time systems, by matrices 
$A \in \Rmn{n}{n}$, $B \in \Rmn{n}{p}$, $C \in \Rmn{m}{n}$ and $D \in \Rmn{m}{p}$ as

\begin{align}
\dot{x} & =  Ax + Bw
\label{eqn:AB}
\\
z & =  Cx + Dw
\label{eqn:CD}
\end{align}
where $w$ is a disturbance feedback depending linearly on the output $z$ \cite[p.~538]{HiPr05}.
For simplicity, we restrict our attention to continuous-time systems for most of the paper,
but we briefly explain how to extend our results and methods to discrete-time systems in Section \ref{sec:discrete}.

The real stability radius, sometimes called the real structured stability radius, is a well known 
quantity for measuring robust stability  of linear dynamical systems with output feedback
\cite{HinPri90a,HinPri90b,Hinrichsen93stabilityradii,HiPr05,ZhoGloDoy95,Kar03}. It measures stability under a certain class of real perturbations where the size of the perturbations are measured by a given norm $\norm{\cdot}$. Most of the literature has focused on spectral norm bounded perturbations for which there exists a characterization in terms of an explicit formula \cite{QiuEtAl95} and a level-set algorithm \cite{DoorenRealStabRad96}. This algorithm has been proven to be convergent; however, it is not practical for systems where large and sparse matrices arise as it requires a sequence of Hamiltonian eigenvalue decompositions, each with a complexity of $\bigo(n^3)$. 

As an alternative to the spectral norm, Frobenius-norm bounded perturbations have also been of interest to the control community \cite{LeeStab96,BoukasShi99,BoukasShi98,BulRealStabRad01,BobRealStabRad99, BobRealStabRad01}. It has been argued that the Frobenius norm is easier to compute and is more advantageous to consider in certain types of control systems \cite{BobRealStabRad99, BobRealStabRad01},
admitting natural extensions to infinite-dimensional systems \cite{BulRealStabRad01}. In the special case $B=C=I, D=0$, there exists an algorithm  \cite{BobRealStabRad99, BobRealStabRad01} that gives upper and lower bounds for the Frobenius-norm bounded real stability radius; however, there is no algorithm to our knowledge that is applicable in the general case. Indeed, \cite{BeVo13} describes this as an unsolved research problem. In this paper, we present the first method to our knowledge that provides good
approximations to the Frobenius-norm bounded real stability radius.  

Our method relies on two foundations. The first is the theory of spectral value sets associated with the dynamical system \eqref{eqn:AB}--\eqref{eqn:CD}
as presented in \cite[Chapter 5]{HiPr05}.  The second is the appearance of a 
number of recent iterative algorithms that find rightmost points of spectral value sets of various sorts, 
beginning with the special case of matrix pseudospectra (the case $B=C=I, D=0$) \cite{GugOve11}, followed by a related method for 
pseudospectra \cite{KressVander12} and extensions to real-structured pseudospectra \cite{GuLu13,GugMan14,Rost15,Gug16}, and to 
spectral value sets associated with \eqref{eqn:AB}--\eqref{eqn:CD} \cite{GuGuOv13,MitOve15} and 
related descriptor systems \cite{BeVo13}. 

The paper is organized as follows. In Section \ref{sec:fundamentals} we introduce spectral value sets,
establishing a fundamental relationship between the spectral value set abscissa and the stability
radius.  In Section \ref{sec:odeframework} we
introduce an ordinary differential equation whose equilibria are generically associated with rightmost
points of Frobenius-norm bounded real spectral value sets.
In Section \ref{sec:method} we present a practical iterative method to compute these points.
This leads to our method for approximating the Frobenius-norm bounded real stability radius, presented in 
Section~\ref{sec:stabrad}.  We outline extensions to discrete-time systems in Section~\ref{sec:discrete}, 
present numerical results in Section~\ref{sec:numerical}, and make concluding remarks in Section~\ref{sec:conclusion}.

\section{Fundamental Concepts} 
\label{sec:fundamentals}

Throughout the paper, $\norms{\cdot}$ denotes the matrix 2-norm (maximum singular value), whereas $\normf{\cdot}$ denotes the Frobenius norm (associated with the trace inner product). The usage $\norm{\cdot}$ means that the norm may be
either  $\norms{\cdot}$ or $\normf{\cdot}$, or both when they coincide, namely for vectors or rank-one matrices.
We use the notation $\olhp$ to denote the open \emph{left} half-plane $\{\zlam: \Re(\zlam)<0\}$ 
and $\cuhp$  to denote the closed \emph{upper} half-plane $\{\zlam: \Im(\zlam)\geq 0\}$.

\subsection{Spectral Value Sets and $\mu$-Values}
\label{subsec:specvalsets}

Given real matrices $A,B,C$ and $D$ defining the linear dynamical
system \eqref{eqn:AB}--\eqref{eqn:CD}, linear feedback $w=\Delta \, z$ leads to a \emph{perturbed system matrix}
with the linear fractional form
\beq\label{pertsysmatdef}
    M(\Delta) = \eigtfmat \quad \mathrm{for~} \Delta \in \K^{p\times m}
\eeq
where the field $\K$ is either $\R$ or $\C$.  Note that since
\beq
    \|D\Delta\|_2 \leq \|D\|_2\|\Delta\|_2  \leq \|D\|_2\|\Delta\|\fro 
 \label{normrelation}
\eeq
we can ensure that $M(\Delta)$ is well defined by assuming $\|\Delta\|\leq\eps$ and $\eps\|D\|_2 < 1$,
regardless of whether $\|\Delta\|$ is $\|\Delta\|_2$ or $\|\Delta\|\fro$.

\begin{defi}
Let $\eps\in \R$, with $\eps\|D\|_2 < 1.$  Define the \emph{spectral value set} with respect to the norm $\norm{\cdot}$ and the field $\K$ as
\beqs
            \SVSKepsn(A,B,C,D) = \bigcup \left\{\spec(M(\Delta)) : \Delta\in\K^{p\times m},
						\|\Delta\| \leq \eps\right\}.
\label{SVSepsR_spec}
\eeqs
\end{defi}
Here $\spec$ denotes spectrum.  Note that 
\[
    \SVSKepss(A,B,C,D)\supseteq \SVSKepsf(A,B,C,D) \supseteq \SVSKzeron(A,B,C,D) = \spec(A).
\]

It is well known that when $\K=\C$, the spectral value set can equivalently be defined as the set of
points $s\in\C$ for which the spectral norm, i.e., the largest singular value, of the transfer matrix 
\beqs
    G(s) = C(s I-A)^{-1}B+D
\eeqs
takes values at least $1/\eps$ \cite[Chap. 5]{HiPr05}. Furthermore, it is well known that 
\beq
	\SVSCepsn(A,B,C,D) = \bigcup \left\{\spec(M(\Delta)) : \Delta\in\Cmn{p}{m},
						\|\Delta\| \leq \eps \mathrm{~and~}\rank(\Delta)\leq 1 \right\}.
\label{C-rank-one}
\eeq
As a consequence of this rank-one property, it is clear that 
\[
     \SVSCepss(A,B,C,D)=\SVSCepsf(A,B,C,D).
\label{SVSC_same_both_norms}
\]
In contrast, when perturbations are restricted to be real, the inclusion 
\[
	\SVSRepss(A,B,C,D)\supseteq\SVSRepsf(A,B,C,D)
\]
is generically strict.  Instead of ordinary singular values, we must consider \emph{real structured singular values} or
\emph{real $\mu$-values}, that is with perturbations restricted to real matrices.
This topic is discussed at length in \cite[Section 4.4]{HiPr05},
allowing additional structure to be imposed on $\Delta$ beyond simply $\Delta\in\K^{p\times m}$, and treating
a general class of operator norms, including the spectral norm, but not, however, the Frobenius norm. See also \cite[Chapter 6] {Kar03}.

\begin{defi} \label{def:muvalue}
The $\mu$-value of a matrix $H \in \Cmn{m}{p}$ with respect to the field $\K$ and the norm $\norm{\cdot}$ is defined by
\beq
\muk(H) = \left[ \inf \left\lbrace \norm{\Delta} ~:~  \Delta \in\K^{p \times m} ,~\det(I-H \Delta) = 0 \right\rbrace \right]^{-1}.
\eeq
\end{defi}
We use the convention that taking the infimum over the empty set always
yields $\infty$ and that $\infty^{-1}=0$, so that $\muk(0)^{-1}=\infty$.
Definition \ref{def:muvalue} defines the \emph{real $\mu$-value} when $\K = \R$
for complex matrices $H$ as well as real matrices.
 
The following lemma is well known in the case of the spectral norm, where it
is usually known as the Eckart-Young theorem.  
See \cite[Lemma 1]{QiuEtAl95} and \cite[Prop. 4.4.11]{HiPr05} for extensions to other structures and other operator
norms.  A key point to note here is that while this result holds both for $\K=\C$ and $\K=\R$, it does \emph{not} hold
for the real $\mu$ value when $H$ is complex. 

\begin{lemm}\label{lem:mu-D-equality} Let $H\in\K^{m \times p}$ and let $\norm{\cdot}$ be either $\norms{\cdot}$ or $\normf{\cdot}$. 
Then, $ \muk(H) = \|H\|_2.$
\end{lemm}
\begin{proof}
If $H=0$, the result is clear. Assume $H\neq 0$. If $ \|\Delta\| < \|H\|^{-1}_2 $, then we have from \eqref{normrelation}
that $\det(I-H\Delta)\neq0$. This shows that $\muk(H)\leq \|H\|_2$.  For the reverse inequality,
let $H$ have the singular value decomposition $U\Sigma V\tp$ where $U$ and $V$ are unitary and
$$\Sigma = \diag\{\sigma_1(H),\sigma_2(H),\ldots,\sigma_{\min\{p,m\}}(H)\}$$
is a diagonal matrix with singular values on the diagonal in descending order by magnitude. Define
$$\Delta = V \diag \{ \sigma_1(H)^{-1},0,\ldots,0 \}U\tp.$$
Then $\|\Delta\|\fro = \|\Delta\|_2= \|H\|_2^{-1}$ and $\det(I-H\Delta)=0$. Furthermore, if $\K=\R$ then
since $H$ is real, $\Delta$ is also real. This shows that $\muk(H)\geq \|H\|_2$.
\end{proof}

Combining \cite[Lemma 5.2.7]{HiPr05} with Lemma \ref{lem:mu-D-equality} results in the following corollary.  This may be compared
with  \cite[Corollary 5.2.8]{HiPr05}, which treats a more general class of structured perturbations, but is restricted to operator norms.
\begin{coro}\label{coro:eps-as-mu-levelsets}
Let $\norm{\cdot}$ be either $\norms{\cdot}$ or $\normf{\cdot}$.
Let $s\in \C\backslash\spec(A)$ and $ \norm{D} < \muk\left(G(s)\right)$ and define
$$ \eps(s) = \min \left\lbrace \norm{\Delta} ~:~ \Delta \in \K^{p \times m}, s \in \spec(M(\Delta)) \right\rbrace.$$
Then, 
$$ \eps(s) = \left (\muk\left(G(s)\right)\right)^{-1}.$$
\end{coro}

This leads to the following theorem which may be compared to \cite[Theorem 5.2.9]{HiPr05}, which again
does not treat the Frobenius norm.
\begin{theo}\label{theo:svsa-vs-real-mu} Let $\norm{\cdot}$ be either $\norms{\cdot}$ or $\normf{\cdot}$.
Suppose $\eps>0$ and $\eps\norms{D}<1$. Then
\beqs
    \SVSKepsn(A,B,C,D) = \spec(A) ~\bigcup~ \left\lbrace s\in \C\backslash\spec(A) ~:~ \muk \left(G(s)\right) \geq \eps^{-1}   \right\rbrace.
\eeqs
\end{theo}

\begin{proof} 
Suppose $s \in \sigma(M(\Delta)) ~\cap~ \left\lbrace s \in \C\backslash\spec(A)\right\rbrace$, $\Delta \in \K^{p \times m}$ and $\norm{\Delta} \leq \eps$. By \cite[Lemma 5.2.7]{HiPr05}, we have $\left(\muk(G(s))\right)^{-1} \leq \norm{\Delta}\leq \eps$. Conversely, if $s\in \C\backslash\spec(A)$ and $\left(\muk\left(G(s)\right)\right) \geq \eps^{-1}$, then we have $\muk\left(G(s)\right) > \norms{D}$ and by Corollary \ref{coro:eps-as-mu-levelsets}, there exists $\Delta$ with $\norm{\Delta} = \left(\muk\left(G(s)\right)\right)^{-1} = \eps $ 
such that $s \in \spec(M(\Delta))$.
\end{proof}

Note that even when $A,B,C,D$ are real, the transfer function $G(s)$ is normally complex for $s\not\in\R$ so
it is not generally the case that $\mur(G(s)) = \norms{G(s)}$.
For real spectral value sets defined by the spectral norm, the optimal perturbation that appears in Definition \ref{def:muvalue} of the $\mu$-value can in fact always be chosen to have rank at most two \cite[Section 2]{QiuEtAl95}, leading to the formula
\beqs
            \SVSRepss(A,B,C,D) = \bigcup \left\{\spec(M(\Delta)) : \Delta\in\Rmn{p}{m},
						\|\Delta\|_2 \leq \eps\mathrm{~and~}\rank(\Delta)\leq 2 \right\}.
\label{R-rank-two}
\eeqs
We make the assumption in this paper that the same property holds for the Frobenius norm, but, for brevity,
we leave a detailed justification of this to future work.

\subsection{The Stability Radius}
\label{subsec:stabrad}
Because we are focusing on the continuous-time dynamical system \eqref{eqn:AB}--\eqref{eqn:CD}, the stability
region of interest is the open left half-plane $\olhp$.  We say that $A$ is \emph{stable} if
$\spec(A)\in\olhp$, in which case, for sufficiently small $\eps$, the spectral value set $\SVSKepsn(A,B,C,D)$ is also in $\olhp$.  
The stability radius $\stabradKn$ measures the size of  the minimal perturbation that destabilizes the matrix or results 
in $M(\Delta)$ being undefined \cite[Def.~5.3.1]{HiPr05}.

\begin{defi} The stability radius $\stabradKn$ is defined with respect to the field $\K$ and the norm $\norm{\cdot}$ as
\[
\stabradKn(A,B,C,D) = \inf \{ \| \Delta \| :\Delta\in\K^{p\times m} , \det(I-D \Delta)=0 ~ \mbox{or} ~ \sigma(M(\Delta)) \not\subset\olhp \}.
\]
\label{def:real-stab-rad-orig}
\end{defi}

The characterization 
\beq\label{def:real-stab-rad}
\stabradKn(A,B,C,D) =  \min \left( \left[\mu_\K^{\norm{\cdot}}(D)\right]^{-1}, \inf_{\omega \in \R} \left[ \mu_{\K}^{\norm{\cdot}}\left(G(\iu \omega)\right)\right]^{-1} \right) 
\eeq 
is well known for operator norms \cite[Theorem 5.3.3]{HiPr05}.  Corollary \ref{coro:eps-as-mu-levelsets} and Lemma \ref{lem:mu-D-equality}
extend \cite[Theorem 5.3.3]{HiPr05} beyond operator norms to the Frobenius norm leading to a similar formula 
\beq\label{equiv:real-stab-rad}
     \stabradKn(A,B,C,D)  = \min \left( \norms{D}^{-1}, \inf_{\omega \in \R} \left[ \mu_{\K}^{\norm{\cdot}}\left(G(\iu \omega)\right)\right]^{-1} \right).
\eeq
\begin{rema}\label{rema-upper-semicont} As the $\mu_{\R}^{\norm{\cdot}}$ function is upper semi-continuous both for operator norms and the Frobenius norm (see \cite[Lemma 1.7.1]{Kar03}), we have $\lim_{| \omega | \to  \infty } G(iw) $ $= D$ but 
    \beq \liminf_{| \omega | \to\infty} \left[ \mu_{\K}^{\norm{\cdot}}\left(G(\iu \omega)\right)\right]^{-1} \geq  \left[\mu_\K^{\norm{\cdot}}(D)\right]^{-1} = \norms{D}^{-1}
		\label{upper-semicont-mu}   
   \eeq
with a possible strict inequality (see \cite[Remark 5.3.17 (i)]{HiPr05} and \cite[Example 5.3.18]{HiPr05} for an example with $p=1$). Therefore, when $D \neq 0$, we cannot eliminate the first term in \eqref{equiv:real-stab-rad}. Either $\stabradKn(A,B,C,D) = \norms{D}^{-1}$ or the infimum in \eqref{equiv:real-stab-rad} is strictly less than $\norms{D}^{-1}$ in which case it has to be attained at a finite $\omega$; otherwise, we would obtain a contradiction as $|\omega | \to \infty$ by the inequality \eqref{upper-semicont-mu}. However, in the special case when $D=0$, we can interpret $ \left[\mu_\K^{\norm{\cdot}}(D)\right]^{-1} = \|D\|_2^{-1}=\infty$ (see the paragraph after Definition \ref{def:muvalue}) and dispense with the first term in \eqref{equiv:real-stab-rad}. 
 \end{rema}

In the complex case $\K=\C$, the spectral norm and Frobenius norms define the same stability radius $\stabradCn$.
In this case also we can eliminate the first term in \eqref{equiv:real-stab-rad}, since \eqref{upper-semicont-mu} holds
with equality, and the second term is simply the reciprocal of the $\Hinf$ norm of the transfer matrix $G$ on the boundary 
of the stability region.
The standard method to compute it is the Boyd-Bala\-krish\-nan-Bru\-insma-Stein\-buch (BBBS)
algorithm \cite{Boba90,BruSte90}. This algorithm is globally and quadratically convergent, but is not practical when 
$n$ is large due to its computational complexity: it requires repeated computation of all eigenvalues of
$2n\times 2n$ Hamiltonian matrices. The first constructive formula to compute $\murs$ and hence $\stabradRs$, the real $\mu$ value and the real stability radius
for the spectral norm, was given in \cite{QiuEtAl95};
this led to a practical level-set algorithm \cite{DoorenRealStabRad96}.  However, this is significantly more involved than the
BBBS algorithm and hence is also impractical in the large-scale case. To our knowledge, no efficient algorithm to compute $\stabradRs$ is known when $n$ is large.  
As noted in the introduction, much less attention
has been given to the Frobenius-norm case, though it is clearly of interest in applications. As far as we know, no constructive
method has been given to approximate $\murf$ or $\stabradRf$, even if $n$ is small.

\subsection{The Spectral Value Set Abscissa} \label{subsec:specvalsetabsc}
The spectral value set abscissa measures how far the spectral value set extends rightwards into the complex plane
for a prescribed value of $\eps$. 
\begin{defi}
For $\eps\geq 0$, $\eps\|D\|_2<1$, the \emph{spectral value set abscissa} (w.r.t.\ the norm $\|\cdot\|$ and the field $\K$) is
\beq  
         \alepsKn(A,B,C,D) =  \max\{\mathrm{Re}~\zlam : \zlam \in \SVSKepsn(A,B,C,D)\} \label{alepsdef}
\eeq
with $\alpha^{\K,\|\cdot\|}_0(A,B,C,D)=\alpha(A)$, the spectral abscissa of $A$. 
\end{defi}
In the case $\K=\R$, $\alepsRn(A,B,C,D)$ is
called the \emph{real spectral value set abscissa}. 

\begin{defi}\label{def:locright}
A \emph{rightmost}  
point of a set $S\subset\C$ is a point where the maximal value of 
the real part 
of the points in $S$ is attained. A \emph{locally rightmost}  
point of a set 
$S\subset\C$ is a point $\zlam$ which is a
rightmost 
point of $S\cap{\mathcal N}$ for some neighborhood ${\mathcal N_\delta}=\{s:|s-\zlam|<\delta\}$ of $\zlam$
with $\delta>0$.
\end{defi}

\begin{rema}
Since $\SVSKepsn(A,B,C,D)$ is compact, its rightmost points, that is the maximizers of the
optimization problem in \eqref{alepsdef} 
lie on its boundary. When $\K=\C$, there can only be a finite number 
of these \cite[Remark 2.14]{GuGuOv13}. However, when $\K=\R$, there can be an infinite number of points on the boundary with 
the same real part, 
as shown by the following example. 

\begin{example}
\label{mertexample2}
Let
\[
A =
\left(
\begin{array}{cc}
  0 &  -1  \\
  1 &  ~~0     
\end{array}
\right),
\quad
B = \left(
\begin{array}{l}
0  \\
1 
\end{array}
\right), 
\quad
C = \left(
\begin{array}{cc}
1 & 0 
\end{array}
\right), 
\quad
D = 0.
\]
For $\eps\in(0,1)$, $\SVSReps(A,B,C,D)$ consists of two line segments on the imaginary axis, which merge into one line segment
$\left [\frac{\sqrt{2}}{2}\iu, -\frac{\sqrt{2}}{2}\iu\right ]$ when $\eps=1$.
\end{example} 
\end{rema}

\begin{rema}\label{rema:realsym}
Since $A,B,C,D$ are real, $\SVSKepsn(A,B,C,D)$ is symmetric with respect to the real axis, so
without loss of generality, when we refer to a rightmost 
$\zlam$ in $\SVSKepsn$, we imply that 
$\zlam\in\cuhp$, the closed upper half-plane,
and when we say that the rightmost point $\zlam$ is unique, we mean considering only points in $\cuhp$. 
The same convention applies to the spectrum, so that a rightmost  
eigenvalue is understood to be in $\cuhp$.
\end{rema}

There is a key relationship between the spectral value set abscissa and the stability radius
that is a consequence of Theorem \ref{theo:svsa-vs-real-mu}:
\begin{coro}
\begin{align}
\label{def:real-stab-rad-2}
\stabradKn(A,B,C,D) &=  \inf \left\lbrace \eps :~ \eps\|D\|_2 < 1 \mbox{~or~} \alepsKn(A,B,C,D) \geq 0 \right\rbrace
\\
\label{def:real-stab-rad-3}
                                 &=  \min \left( \norms{D}^{-1}, \inf \left\lbrace \eps :~ \alepsKn(A,B,C,D) \geq 0 \right\rbrace \right). 
\end{align}
\end{coro}
\begin{proof}
That the right-hand sides of \eqref{def:real-stab-rad-2} and \eqref{def:real-stab-rad-3} are the same
is immediate.
Hence, it suffices to show that both of the infimum terms in \eqref{equiv:real-stab-rad} and \eqref{def:real-stab-rad-3} are attained and are equal when the upper bound $\| D\|_2^{-1}$ is not active. The infimum in \eqref{def:real-stab-rad-3} is attained because $\alepsKn(A,B,C,D)$ is a  monotonically increasing continuous function of $\eps$ (see \cite[Chapter 2]{Kar03} for continuity properties of real spectral value sets) and the infimum in  \eqref{equiv:real-stab-rad} is attained at a finite point by Remark \ref{rema-upper-semicont}. Finally, the infimal values are equal by Theorem \ref{theo:svsa-vs-real-mu}.
\end{proof}

The algorithm developed in this paper for approximating the real stability radius $\stabradRf$ when $n$ is large depends on the
fundamental characterization \eqref{def:real-stab-rad-3}. This was also true of the recent algorithms
developed in \cite{GuGuOv13,BeVo13,MitOve15} for the complex stability radius $\stabradCn$ when $n$ is large, for which the
equivalence \eqref{def:real-stab-rad-3} is more straightforward and well known.\footnote{For a different approach to approximating
$\stabradCn$ when $n$ is large, namely the ``implicit determinant" method, see \cite{FreSpeVan14}.}

\section{An Ordinary Differential Equation}
\label{sec:odeframework}

This section extends the method of Guglielmi and Lubich \cite[Sec.~2.1]{GuLu13} for approximating the real pseudospectral abscissa
(the real spectral value set abscissa in the case $B=C=I$, $D=0$) to the spectral value set abscissa $\alepsRf$ for general
$A,B,C,D$, using the Frobenius norm. As we shall see, the extension is not straightforward as additional subtleties arise
in the general case that are not present in the pseudospectral case. 

We consider $\eps>0$ to be fixed with $\eps\norms{D}<1$ throughout this section and the next section.
We start by considering the variational behavior of eigenvalues of the perturbed system matrix $M(\Delta)$ defined in \eqref{pertsysmatdef}.
It is convenient to assume a smooth parametrization $t\mapsto \Delta(t)$ mapping  $\R$ to $\Rmn{p}{m}$,  
with $\|\Delta(t)\|\fro = \eps$ for all $t$.  We will use $\dot\Delta$ to denote the derivative $(d/dt) \Delta(t)$. 

We need the following lemma.

\begin{lemm} Given a smooth parametrization $\Delta(t)$ with $\|\Delta(t)\|\fro=1$,
we have
\begin{equation}
\frac{d}{d t} \biggl( \Delta(t) \left( I - D \Delta(t) \right)^{-1} \biggr)=
\left( I - \Delta(t) D \right)^{-1} \dot \Delta(t)  \left( I - D \Delta(t) \right)^{-1}.
\end{equation}
\label{lem:derE}
\end{lemm}

\begin{proof}
For conciseness, we omit the dependence on $t$, differentiate and regroup terms as
\begin{align}\label{eq:matrix-Ft-deriv}
\frac{d}{d t} \biggl( \Delta \left( I - D \Delta \right)^{-1} \biggr) &= \dot{\Delta}  \left( I - D \Delta \right)^{-1} +
\Delta \frac{d}{d t} \left( I - D \Delta \right)^{-1} \nonumber \\ 
&= \dot{\Delta} \left( I - D \Delta \right)^{-1} +
\Delta \left( I - D \Delta \right)^{-1} D\dot{\Delta} \left( I - D \Delta \right)^{-1} \nonumber \\
&= \left(I+ \Delta \left( I - D \Delta \right)^{-1} D \right) \dot{\Delta} \left( I - D \Delta \right)^{-1}.
\end{align}
We then observe that
\begin{equation}\label{eq:matrix-inf-series}
I+ \Delta \left( I - D \Delta \right)^{-1} D = I + \Delta \biggl( \sum\limits_{k=0}^{\infty} (D \Delta)^k \biggr) D
= I + \sum\limits_{k=1}^{\infty} (\Delta D)^k = \left( I - \Delta D \right)^{-1}.
\end{equation}
Combining \eqref{eq:matrix-Ft-deriv} and \eqref{eq:matrix-inf-series} yields the result.
\end{proof}

The following definition from \cite{GugOve11,MitOve15} is useful.

\begin{defi}\label{def:eigentriple}
Let $\lambda$ be a simple eigenvalue of a matrix $M$ with associated right eigenvector $x$ satisfying $Ax=\lambda x$
and left eigenvector $y$ satisfying $y^* M=\lambda y^*$.  We refer to $(\lambda$, $x$, $y)$ as an \emph{RP-compatible} eigentriple of $M$
if $y^*x$ is real and positive and $\|x\|=\|y\|=1$, and as a
rightmost eigentriple if $\lambda$ is a rightmost eigenvalue of $M$ in $\cuhp$.
\end{defi}

Note that if $(\lambda,x,y)$ is an RP-compatible eigentriple of $M$, so is $(\lambda$, $e^{\imagunit\theta}x$ ,$e^{\imagunit\theta}y)$
for any $\theta\in[0,2\pi)$.

Then we have:

\begin{lemm}  
\label{lem:lambdaderiv}  Given a smooth parametrization $\Delta(t)$ with $\|\Delta(t)\|\fro = \eps$,
let $\lambda(t)$ be a continuously varying simple eigenvalue of
\[
      M\left(\Delta(t) \right) = A + B \Delta(t) \left( I- D \Delta(t) \right)^{-1}C.
\] 
Then $\lambda(t)$ is differentiable with
\[
            \Re\dot{\lambda}(t) = \frac{1}{y(t)^*x(t)}\Re(u(t)^* \dot{\Delta}(t) v(t))
\label{eq:optprob2}
\]
where $(\lambda(t)$, $x(t)$, $y(t))$ is an RP-compatible eigentriple of $M(\Delta(t))$ and
\[
      u(t) = \left( I - \Delta(t) D \right)\tp[-] B\tp y(t), \quad v(t) = \left( I -  D \Delta(t) \right)^{-1} C x(t).
\]
\end{lemm}

\begin{proof}
Applying standard eigenvalue perturbation theory \cite[Theorem 6.3.12]{HJ90}, 
together with Lemma \ref{lem:derE}, we find that $\lambda$ is differentiable with
\beq
\dot{\lambda} = \frac{y^* \dot{M} \left(\Delta \right) x}{y^* x} \quad \mbox{and} \quad \dot{M} \left(\Delta \right) =
            B \left(I - \Delta D \right)^{-1} \dot{\Delta} \left(I -  D \Delta \right)^{-1} C
\label{eq:optprob}
\eeq
where we omitted the dependence on $t$ for conciseness. The result is then immediate. 
\end{proof}

In what follows next it is convenient to define $E(t)=\Delta(t)/\eps$, so that $\|E(t)\|\fro = 1$ for all $t$.
Consequently, we have
\beq
            \Re\dot{\lambda}(t) = \frac{\eps}{y(t)^*x(t)}\Re\left(u(t)^* \dot{E}(t) v(t)\right) =
             \frac{\eps}{y(t)^*x(t)}\left \langle \Reuvt, \dot{E}(t) \right\rangle,
 \label{eq:lamderiv}
\eeq
where for $R,S \in \Rmn{p}{m}$, 
\[
    \langle R, S \rangle = \Tr R\tp S = \sum\limits_{i,j} R_{ij}S_{ij},
\]
the trace inner product on $\Rmn{p}{m}$ associated with the Frobenius norm.
The condition that $\|E(t)\|\fro$ is constant is equivalent to
\begin{equation}\label{norm-preservation-E}
    \frac d{dt} \| E(t) \|\fro^2 = 2 \left\langle E(t), \dot E(t) \right\rangle = 0 \quad \forall t.
\end{equation}
Our aim, given $t$, is to choose $\dot{E}(t)$ to maximize \eqref{eq:lamderiv} subject to the constraint \eqref{norm-preservation-E},
leading to an optimization problem whose solution is given by the following lemma. The proof is a straightforward application
of first-order optimality conditions; see also \cite[Lemma 2.4]{GuLu13}.
\begin{lemm} 
\label{lem:opt} 
Let $E\in\Rmn{p}{m}$  have unit Frobenius norm, 
and let $u \in \C^p,v\in\C^m$ be given complex vectors
such that $\Reuv\neq0$. A solution to the optimization problem
\beq
\tilde Z  =  \arg\max_{Z \in \Omega} \ \Re \left( u^* Z v  \right), \qquad
\Omega = \{ Z \in \Rmn{p}{m}, \|Z\|\fro = 1,\langle E, Z \rangle=0 \}
\label{eq:opt}
\eeq
exists and it satisfies 
\beq
\tau \tilde Z =  \Bigl( \Reuv  - \left\langle E,  \Reuv \right\rangle E\Bigl),
\label{eq:Eopt}
\eeq
where $\tau$ is the Frobenius norm of the matrix on the right-hand side in $(\ref{eq:Eopt})$.
\end{lemm}

This suggests consideration of the following ordinary differential equation (ODE) on the manifold of 
real $p\times m$ matrices of unit Frobenius norm:
\beq\label{ode-E}
\dot E(t) = \Reuvt - \left\langle E(t),  \Reuvt \right\rangle E(t),
\eeq
with $u(t)$ and $v(t)$ defined by
\beq
      u(t) = \left( I - \eps E(t) D \right)\tp[-] B\tp y(t), \quad v(t) = \left( I -  \eps D E(t) \right)^{-1} C x(t)
\label{eq:uvdef}
\eeq
where $(\lambda(t)$, $x(t)$, $y(t))$ is a rightmost RP-compatible eigentriple for the matrix 
$M(\eps E(t))$ (see Definition \ref{def:eigentriple}).
Assume the initial condition $E(0)=E_0$,
a given matrix with unit Frobenius norm, chosen so that $M(\eps E_0)$ has a unique rightmost eigenvalue
(considering only eigenvalues in $\cuhp$),
and that this eigenvalue, $\lambda(0)$, is simple.


\subsection{Equilibrium Points of the ODE}\label{subset:equilibrium}
We now focus on the properties of the ODE \eqref{ode-E}, in particular, characterizing equilibrium points.

\begin{theo} \label{thm:ode}
Let $E(t)$ with unit Frobenius norm satisfy the differential equation {\rm (\ref{ode-E})} initialized as described above.
There exists $t_{\max}\in (0,\infty]$
such that, for all $t\in[0,t_{\max})$ 
\begin{itemize}
\item[{\rm (1)}] The eigenvalue $\lambda(t)$ is the unique rightmost eigenvalue of $M(\eps E(t))$ in $\cuhp$
and this eigenvalue is simple, so the ODE is well defined.
\item[{\rm (2)}] $\|E(t)\|\fro = 1$.
\smallskip
\item[{\rm (3)}]$\Re  \dot{\lambda}(t) \geq 0$.
\end{itemize}
Furthermore, at a given value $t\in[0,t_{\max})$, the following three conditions are equivalent:
  \begin{itemize}
    \item[{\rm (i)}] $\Re \dot \lambda(t)= 0$. 
    \item[{\rm (ii)}]One of the following two mutually exclusive conditions holds:
    \beq
         \Reuvt = 0 \mathrm{~~or~~}  E(t) = \frac{\Reuvt}{\|\Reuvt\|\fro}.  \label{twoconds}  
   \eeq
    \item[{\rm (iii)}] $\dot E(t) = 0$. 
  \end{itemize}
Finally, if the second alternative in \eqref{twoconds} holds at $t_0\in [0,t_{\max}]$, then
there does not exist any locally differentiable path $F(t)$, with $\|F(t)\|\fro=1$ and $F(0)=E(t_0)$,
for which the rightmost eigenvalue of $M(\eps F(t))$, say $\kappa(t)$, has $\Re\dot\kappa(0)>0$.

\end{theo}

\begin{proof} 
(1)  Because the rightmost eigenvalue is unique and simple for $t=0$, the same property must hold for sufficiently small positive $t$,
establishing the existence of $t_{\max}>0$ such that the ODE is well defined for all $t\in[0,t_{\max}]$.
(2) Taking the trace inner product of $E(t)$
with the ODE \eqref{ode-E} we find that \eqref{norm-preservation-E} holds for $\|E(t)\|\fro=1$, so the norm is preserved by the ODE.
(3) Substituting the ODE \eqref{ode-E} into \eqref{eq:lamderiv} we obtain from Cauchy-Schwartz that
\beq
            \Re\dot\lambda(t) = \frac{\eps}{y(t)^*x(t)}\left[ \|\Reuvt\|\fro^2 - \left\langle\Reuvt,E(t)\right\rangle^2\right] \geq 0,  \label{relampos}
\eeq
establishing (i). For (ii), equality holds in \eqref{relampos} at a given $t$ if and only if one of the two alternatives
in \eqref{twoconds} holds.
That (iii) 
is an equivalent condition follows directly from the ODE \eqref{ode-E}. 
The final statement follows from the optimality property of Lemma~\ref{lem:opt}.
\end{proof}

Thus, equilibria of the ODE come in two flavors. When the second alternative in \eqref{twoconds} holds,
a first-order optimality condition for $\lambda(t)$  to be a rightmost point of $\SVSKepsf$ holds, implying in 
particular that it is on the boundary of $\SVSKepsf$.  However, when the first alternative holds,
we cannot make any such claim. In the special case of pseudospectra, that is with $B=C=I$ and $D=0$,
the outer product $u(t)v(t)^*$ reduces to $y(t)x(t)^*$, whose real part \emph{cannot} be zero as was shown in \cite[Sec.\ 2.1.6]{GuLu13}: in fact, the
proof given there shows that $\real{y(t)x(t)^*)}$ has rank one when $\lambda(t)$ is real and rank two when it is complex.

If $\Reuvt=0$, we say that $\lambda(t)$ is a \emph{static point}.
Note that this generalizes the notions of uncontrollability and unobservability, because if $\lambda(t)$ is unobservable,
then $B\tp y(t)=0$, implying $u(t)=0$, while if it is uncontrollable,
then $Cx(t)=0$, implying $v(t)=0$.  Example~\ref{mertexample2} given earlier shows that it is possible that $\Reuvt=0$ 
even if $\lambda(t)$ is controllable and observable. In this example, for all $\eps\in (0,1)$, setting $E_0=\pm 1$, we find that 
$u(0)$ and $v(0)$ are both nonzero but that $\real{u(0)v(0)^*}=0$.
However, we do not know whether it is possible for the solution of the ODE to converge to a static point if it is not
initialized there.

The next lemma explicitly states formulas for $\Reuvt$ and bounds on its rank.

\begin{lemm}
\label{lem:reuv}
Fix $t < t_{\max}$ and let $u \in \C^p, v \in \C^m$ be defined by \eqref{eq:uvdef} for some vectors $y(t)=y$ and $x(t)=x$.
If $\lambda=\lambda(t) \in \R$, then we can choose $y$, $x$, $u$ and $v$ to be real, with $\Re(uv^*)=uv\tp$ having
rank~1.  If $\lambda \not\in \R$,  set 
$X=(\Re\, x, \Im\, x) \in  \Rmn{n}{2}$, $Y=(\Re\, y, \Im\, y ) \in  \Rmn{n}{2}$, so $\Re(yx^*) = YX\tp$.  Then
\[
\Re (u v^*) = \left( I - \eps E D \right)\tp[-] B\tp Y
  X\tp C\tp \left( I - \eps D E \right)\tp[-] 
\]
with
\[
   \mbox{rank}\left( \Reuv \right ) = \mbox{rank} \left( B\tp Y X\tp C\tp \right) \leq 2.
\]
Furthermore, if  $\min(p,m) = 1$, then $\mbox{rank}\left( \Reuv \right ) \leq 1$.
\end{lemm}

\begin{proof} The first statement follows from the definition \eqref{eq:uvdef}, noting that $E$ and $D$ are real.
The rank results follow from submultiplicativity.
\end{proof}

As already mentioned, the argument given in \cite[Sec.\ 2.1.6]{GuLu13} shows that when $\lambda$ is not real,
the matrix $YX\tp$ has rank two, so when $\min(p,m) \geq 2$, we can expect that $UV\tp$ will also have rank two for generic $B$ and $C$. 

If the ODE is initialized so that for all \mbox{$t\geq 0$}, the rightmost eigenvalue $\lambda(t)$
of $M(\eps E(t))$ is unique (considering only eigenvalues in $\cuhp$) and is simple, then we can
take $t_{\max}=\infty$ in Theorem \ref{thm:ode}.  If we further suppose that the number of equilibrium points
of the ODE is finite, then since \eqref{ode-E} is a gradient system 
(with Lyapunov function $\Re \lambda (t)$), we can apply Lasalle's Theorem \cite{LogRya14} allowing us
to state that $E(t)$ converges to an equilibrium point $\tilde E$ and hence $\lambda(t)$ converges
to some $\tilde \lambda$. Suppose that $\tilde\lambda$ is a unique rightmost eigenvalue of $M(\eps \tilde E)$ 
and is simple, with associated RP-compatible eigentriple $(\tilde\zlam$, $\tilde x$, $\tilde y)$, and define
\[
          \tilde u = \left(I - \eps\tilde E D \right)\tp[-] B\tp \tilde y,
          \quad \tilde v = \left(I -  \eps D \tilde E \right)^{-1} C \tilde x,
\]
by analogy with \eqref{eq:uvdef}. 
Then, if $\Re(\tilde u\tilde v^*)\not =0$, we have, by taking limits in Theorem \ref{thm:ode},
that $\tilde E = \Re(\tilde u\tilde v^*)/\|\Re(\tilde u\tilde v^*)\|\fro$, and that
there does not exist any locally differentiable
path $F(t)$, with $\|F(t)\|\fro=1$ and $F(0)=\tilde E$, for which
the rightmost eigenvalue of $M(\eps F(t))$, say $\kappa(t)$, has
$\Re\dot\kappa(0) > 0$.

To summarize this section, we have characterized equilibria of the ODE \eqref{ode-E} as those which
have a first-order local optimality property with respect to rightmost points of $\SVSRepsf(A,B,C,D)$,
with exceptions that are apparently nongeneric. A natural idea would be to
attempt to approximate $\alepsRf(A,B,C,D)$ by integrating the ODE \eqref{ode-E} numerically to determine its equilibria, 
guaranteeing monotonicity by step-size control.  Such a method would generally find locally rightmost points of
$\SVSRepsf(A,B,C,D)$, although it could not be guaranteed to find globally rightmost points.
However, a serious drawback of this approach is the fact that the solution $E(t)$ (and hence most 
likely its discretization) does not preserve the low rank-structure even if both the initial point $E_0$ and the limit of
$E(t)$ as $t\goes \infty$ both have rank two. Although it is possible to consider an ODE defined on the manifold
of rank-two matrices, as done in \cite{GuLu13} and \cite{GugMan14} for the special case $B=C=I$, $D=0$,
we instead develop an efficient discrete iteration that
is nonetheless based on the ODE \eqref{ode-E}.

\section{An Iterative Method to Approximate the Frobenius-norm Real Spectral Value Set Abscissa}
\label{sec:method}

As in the previous section, assume $\eps$ is fixed with $\eps\norms{D}<1$.
Following an idea briefly mentioned in \cite{GugMan14}, we 
consider the following implicit-explicit Euler discretization of \eqref{ode-E} with a variable step-size $h_k$:
\beq   \label{difference_eqn}
E_{k+1} = E_{k} + h_{k+1} \Bigl( \Re(u_{k+1} v_{k+1}^*) - \big \langle E_{k+1}, \Re(u_{k+1} v_{k+1}^*)\big \rangle E_{k} \Bigr),
\eeq
where
\[
      \label{eq:unvn}
      u_{k+1}  =  \bigl( I - \eps E_{k} D \bigr)\tp[-] B\tp y_{k},  \quad
      v_{k+1} =  \bigl( I - \eps D E_{k} \bigr)^{-1} C x_{k}
\]
and $(\zlam_{k}$, $x_{k}$, $y_{k})$ is a rightmost RP-compatible eigentriple of $M(\eps E_{k})$.
The method is clearly consistent and converges with order $1$ with respect to $h_k$.  

\begin{lemm}
Let $u_0$, $v_0$ be given complex vectors with $\Reuvo \not = 0$ and set $E_0=\Reuvo/\|\Reuvo\|\fro$.
Let $h_k=1/\|\Reuvk\|\fro$. Then the difference equation \eqref{difference_eqn} has the solution
\beq
          E_{k+1} = \frac{\Real{u_{k+1}v_{k+1}^*}}{\|\Real{u_{k+1}v_{k+1}^*}\|\fro}   \label{difference_eqn_sol}
\eeq
as long as the rightmost eigenvalue of $M(\eps E_{k})$ is simple and unique
(considering only those in $\cuhp$) and as long as $\Real{u_{k+1}v_{k+1}^*} \not = 0$, for all $k=0,1,\ldots$.
\end{lemm}
\begin{proof} 
The result is easily verified by substituting \eqref{difference_eqn_sol} into \eqref{difference_eqn}.
The assumptions ensure that the difference equation is well defined.
\end{proof}

Equivalently, let $E_k = U_kV_k\tp$ be the current perturbation, with $\|U_kV_k\tp\|\fro = 1$, and 
$(\zlam_k,x_k,y_k)$ be a rightmost eigenvalue of $M(\eps U_k V_k\tp)$.  
Then by setting
\beq
   X_{k}=(\Re\, x_{k}, \Im\, x_{k}), \
   Y_{k}=(\Re\, y_{k}, \Im\, y_{k} ) 
\nonumber
\eeq
we can write \eqref{difference_eqn_sol} in the form
\beq
           E_{k+1} = U_{k+1} V_{k+1}\tp \mathrm{~with~} \|U_{k+1} V_{k+1}\tp\|\fro = 1
\label{eq:fpR}
\eeq
where 
\begin{align}
   \widehat U_{k+1} & =  \left( I - \eps U_kV_k\tp D \right)\tp[-] B\tp Y_{k} , \\
   \widehat V_{k+1} & =  \left( I - \eps D U_kV_k\tp \right)^{-1} C X_{k},\\
   \beta_{k+1} & = \|\widehat U_{k+1} \widehat V_{k+1}\tp\|\fro^{-1},\\
   U_{k+1} &= \sqrt{\beta_{k+1}} \, \widehat U_{k+1}, \\
   V_{k+1} &= \sqrt{\beta_{k+1}} \, \widehat V_{k+1}.
\label{eq:UkVk}
\end{align}
Since $E_{k}=U_{k}V_{k}\tp$ has rank at most two, we can simplify these expressions using the Sherman-Morrison-Woodbury 
formula \cite{GolVan83} as follows:
\begin{align}
\left( I - \eps U_kV_k\tp D\right)^{-1} & = I + \eps U_{k}  \left( I - \eps V_{k}\tp D U_{k} \right)^{-1} V_{k}\tp D,\label{eq:smw_first} \\
\left( I - \eps D U_kV_k\tp \right)^{-1} & =  I + \eps D U_{k}  \left( I - \eps V_{k}\tp D U_{k} \right)^{-1} V_{k}\tp.\label{eq:smw_second}
\end{align}
Note that $I - \eps V_{k}\tp D U_{k}\in\Rmn{2}{2}$ and is invertible since $\eps\|D\|_2 < 1$ by assumption.
The second formula \eqref{eq:smw_second} can be also used to simplify the definition of the perturbed system matrix
in  \eqref{pertsysmatdef} as follows:
\begin{align}
   M(\pertmat_{k}) & =M(\eps E_{k}) = M(\eps U_{k} V_{k}\tp) \nonumber \\
                            & =A+\eps BU_{k}V_{k}\tp(I- \eps D U_{k} V_{k}\tp)^{-1}C \nonumber \\
                            & = A + \eps  B  U_{k}V_{k}\tp(I + \eps D U_{k} \left( I - \eps V_{k}\tp D U_{k} \right)^{-1} V_{k}\tp )C \nonumber \\
                            & = A + \left( \eps B U_{k} \right) \left[ I + \eps \left( V_k\tp D U_k \right) \left( I - \eps V_{k}\tp D U_{k} \right)^{-1} \right] \left( V_k\tp C \right)
                            \label{pertsysmatdetails}
\end{align}
The product $U_{k} V_{k}\tp$ is never
computed explicitly, but retained in factored form, so that the eigenvalues of $M(\eps_{k} U_kV_{k}\tp)$ with largest real part can be computed
efficiently by an iterative method. The Frobenius norm of the product can be obtained using the following equivalence:
\begin{equation}
	 \|UV\tp \|\fro 
			= \left[ \Tr \left( VU\tp UV\tp \right) \right]^{\frac{1}{2}}\nonumber 
		     = \left[ \Tr \left( \left(U\tp U\right) \left(V\tp V \right) \right) \right]^{\frac{1}{2}}
\end{equation}
which requires only inner products to compute the $2\times 2$ matrices $U\tp U$ and $V\tp V$.

As with the spectral value set abscissa (SVSA) iteration for complex valued spectral values sets given in \cite{GuGuOv13}, there is no guarantee that the full update step for the real Frobenius-norm bounded case will satisfy monotonicity, that is, $\Re(\zlam_{k+1}) > \Re(\zlam_{k})$ may or may not hold, where $\zlam_k$ is a rightmost eigenvalue of \eqref{pertsysmatdetails}.  However, the line search approach to make a monotonic variant \cite[Sec.~3.5]{GuGuOv13} 
does extend to the real rank-2 iteration described above, although, as the derivation is quite lengthy \cite[Sec.~6.3.3]{Mit14}, we only outline the essential components here.  Let the pair $U_{k}$ and $V_{k}$ define the current perturbation, with $\| U_{k} V_{k}\tp \|\fro = 1$, and let the pair $U_{k+1}$ and $V_{k+1}$ be the updated perturbation described above, with $\| U_{k+1} V_{k+1}\tp \|\fro = 1$.  Consider the evolution of a continuously varying simple rightmost eigenvalue $\zlam(t)$ defined on $t \in [0,1]$ of the perturbed system matrix.  The interpolated perturbation is defined using $U_{k}V_{k}\tp$ and $U_{k+1}V_{k+1}\tp$ such that the interpolated perturbation also has unit Frobenius norm, that is, $\zlam(t)$ is an eigenvalue of
\beq 
	\label{eq:m_linesearch}
	M(\pertmat(t)) = A + B \pertmat(t) (I - D\pertmat(t))^{-1}C,
\eeq
where
\beq
	\label{eq:pertmat_linesearch}
	\pertmat(t)=\frac{\eps U(t)V(t)\tp}{\| U(t)V(t)\tp\|}
\eeq
and $U(t) \coloneqq t U_{k+1} + (1-t) U_{k}$ and $V(t) \coloneqq t V_{k+1} + (1-t) V_{k}$.  In \cite[Sec.~6.3.3]{Mit14} 
it is shown that as long $\real{\zlam'(0)} = 0$ does not hold, then $\real{\zlam'(0)} > 0$
can be ensured, though it may requiring flipping the signs of both $U_k$ and $V_k$.  
This result allows a line search to be employed to find a $t \in (0,1)$ such that 
$\Re(\zlam_{k+1}) > \Re(\zlam_{k})$ is guaranteed in an actual implementation of the iteration.  
We now have all the essential pieces necessary to describe approximating $\alepsRf(A,B,C,D)$; 
these are given in Algorithm SVSA-RF.

\begin{algfloat}
\begin{algorithm}[H]
\floatname{algorithm}{}
\caption*{\textbf{Algorithm SVSA-RF}: (Spectral Value Set Abscissa: Real Frobenius-norm)}
\label{alg:rfsvsa}
\begin{algorithmic}[1]
	\algorithmicpurpose{
		to approximate $\alepsRf(A,B,C,D)$
	}
	\REQUIRE{  
		$\eps\in (0,\|D\|_2^{-1})$, $U_0 \in \Rmn{n}{p}$ and $V_0 \in \Rmn{n}{m}$, 
		such that $\|U_0V_0\tp\|\fro = 1$, along with
		eigentriple $(\zlam_0,x_0,y_0)$, with $\zlam_0$ a rightmost eigenvalue of $M(\eps U_0 V_0\tp)$
	}
	\ENSURE{ 
		final iterates $U_k$, $V_k$ with $\|U_kV_k\tp\|\fro = 1$ along with 
		$\zlam_k$, a rightmost eigenvalue of $M(\eps U_k V_k\tp)$, certifying that 
		$\real{\zlam_k} \leq \alepsRf(A,B,C,D)$\\ \quad \\
	}
	
	\FOR {$k = 0,1,2,\ldots$} 
		\STATE \COMMENT{Compute the new perturbation}
		\STATE $X_k \coloneqq (\Re\, x_k, \Im\, x_k)$
		\STATE $Y_k \coloneqq (\Re\, y_k, \Im\, y_k )$
		\STATE $\widehat U_{k+1} \coloneqq
				\left ( I + \eps U_k \left( I - \eps V_{k}\tp D U_{k} \right)^{-1} V_k \tp D \right )\tp B\tp Y_k $
		\STATE $\widehat V_{k+1} \coloneqq  
				\left ( I + \eps D U_k \left( I - \eps V_{k}\tp D U_{k} \right)^{-1} V_k\tp \right ) C X_k$
		\STATE \COMMENT{Normalize the new perturbation}
		\STATE $\beta_{k+1} \coloneqq 
				\left[ \Tr \left( \left(\widehat U_{k+1}\tp \widehat U_{k+1}\right) 
				\left(\widehat V_{k+1}\tp \widehat V_{k+1} \right) \right) \right]^{-\frac{1}{2}}$ 
		\STATE $U_{k+1} \coloneqq \sqrt{\beta_{k+1}} \, \widehat U_{k+1}$
		\STATE $V_{k+1} \coloneqq \sqrt{\beta_{k+1}} \, \widehat V_{k+1}$
		\STATE \COMMENT{Attempt the full update step and, if necessary, do a line search}
		\STATE $(\zlam_{k+1}, x_{k+1},y_{k+1}) \coloneqq 
				\mathrm{a~rightmost~eigentriple~of~} M(\eps U_{k+1}V_{k+1}\tp) $ using \eqref{pertsysmatdetails}
		\IF { $\real{\zlam_{k+1}} \le \real{\zlam_{k}}$ }
			\STATE 
			Find new $\zlam_{k+1}$ via line search using \emph{\eqref{eq:m_linesearch}}
					to ensure $\real{\zlam_{k+1}} > \real{\zlam_{k}}$ 
		\ENDIF
	\ENDFOR
\end{algorithmic}
\end{algorithm}
\algnote{
The $k$th step of the iteration is well defined if $U_{k}V_{k}\tp$ is nonzero and the rightmost eigenvalue of \eqref{pertsysmatdetails} in $\cuhp$ is unique and simple.
}
\end{algfloat}

To summarize this section, we have proposed an efficient method, Algorithm SVSA-RF, to approximate
$\alepsRf(A,B,C,D)$. Although it only guarantees finding
a lower bound on $\alepsRf(A,B,C,D)$,  the ODE \eqref{ode-E} on which it is based has equilibrium points that
typically satisfy a first-order optimality condition (see Section \ref{sec:odeframework}). 
The $k$th step of the iteration is well defined as long as the condition $U_k V_k\tp \not = 0$ holds
and the rightmost eigenvalue of $M(\eps U_{k}V_k\tp)$ is unique and simple.

\section{Approximating the Real Stability Radius by Hybrid Expansion-Contraction}\label{sec:stabrad}

Recall the relationship between the stability radius $\stabradKn(A,B,C,D)$ and the spectral value set abscissa $\alepsKn(A,B,C,D)$ given in
\eqref{def:real-stab-rad-3}, which we write here for the real Frobenius-norm case:
\begin{equation}
\stabradRf(A,B,C,D) =   \min \left( \norms{D}^{-1}, \inf \left\lbrace \eps :~ \alepsRf(A,B,C,D) \geq 0 \right\rbrace \right).
\label{def:real-stab-rad-3-repeat}
\end{equation}
The interesting case is when the second term is the lesser of these two terms, and for the remainder of the paper
we assume this is the case.
It follows that $\stabradRf(A,B,C,D)$ equals the infimum in \eqref{def:real-stab-rad-3-repeat} and
from Remark \ref{rema-upper-semicont} that this infimum is attained. Hence $\alepsRf(A,B,C,D) = 0$ for $\eps=\stabradRf(A,B,C,D)$.
For brevity, we henceforth use $\epsstar$ to denote the real stability radius $\stabradRf(A,B,C,D)$.

Let 
\begin{equation}\label{gdef}
      g(\eps) = \alepsRf(A,B,C,D).
\end{equation}
We wish to find $\epsstar$, the root (zero) of the monotonically increasing continuous function $g$. However,
we do not have a reliable way to evaluate $g$: all we have is Algorithm SVSA-RF which is guaranteed to return a lower
bound on the true value.  Consequently, if the value returned is negative we have no assurance that its sign is correct.
On the other hand, if the value returned is positive, we \emph{are} assured that the sign is correct. This observation
underlies the hybrid expansion-contraction (HEC) algorithm 
recently introduced in \cite{MitOve15} for approximating the complex stability radius, which we now extend to
the real Frobenius-norm case. 

\subsection{Hybrid Expansion-Contraction}\label{subsec:hec}
For any value of $\eps$ satisfying $\epsstar < \eps < \|D\|_2^{-1}$,
there exists a real pertubation matrix $E$ with $\|E\|\fro=1$ such that $M(\eps E)$ has an eigenvalue in the right half-plane.  We assume that $E$ has rank at most two (see the discussion at the end of Section~\ref{subsec:specvalsets}).
See Section~\ref{sec:experiments} for how an initial destabilizing perturbation $\eps UV \tp$ can be found.

Let $U \in \Rmn{p}{2}$ and $V \in \Rmn{m}{2}$ be two matrices such that $\| UV\tp  \|\fro = 1$. Consider
the following matrix family \emph{where $U$ and $V$ are fixed} and $0 < \eps < \norms{D}^{-1}$:
\[ \label{eq:MUVeps}
	M_{UV}(\eps) \coloneqq  M(\eps UV\tp) = A + B \eps UV\tp   (I - D\eps UV\tp  )^{-1} C
\]
and define the function
\begin{equation}\label{gUVdef}
   g_{UV}(\eps) \coloneqq \alpha\left (M_{UV}(\eps) \right),
\end{equation}
the spectral abscissa of $M_{UV}(\eps)$. Unlike $g$, this function is relatively easy to evaluate at a given $\eps$, since 
all that is required is to compute the rightmost eigenvalue of the matrix $M_{UV}(\eps)$, something that we assume
can be done efficiently by an iterative method such as \matlab's  {\tt eigs}, exploiting the equivalence \eqref{pertsysmatdetails}. 
Now, as discussed above, suppose that $\epsub$ is known with $M_{UV}(\epsub)$ having an eigenvalue in the right half-plane.
There exists $\eps_\mathrm{c} \in (0,\epsub)$ such that $g_{UV}(\eps_\mathrm{c})=0$
because $g_{UV}$ is continuous, $g_{UV}(\epsub) > 0$ and $g_{UV}(0) < 0$ (as $A$ is stable).
The contraction phase of the hybrid expansion-contraction algorithm
finds such an $\eps_\mathrm{c}$ by a simple Newton-bisection method, using
the derivative of $g_{UV}(\eps)$ given in Section \ref{subsec:derivs} below.
Note that by definition of $\epsstar$, it must be the case that $\epsstar \leq \eps_\mathrm{c}$.

Once the contraction phase delivers $\eps_\mathrm{c}$ with the rightmost eigenvalue of 
$M_{UV}(\eps_\mathrm{c})$ on the
imaginary axis, the expansion phase then ``pushes"
the rightmost eigenvalue of $M(\eps UV\tp)$ back into the right half-plane using 
Algorithm SVSA-RF, \emph{with $\eps=\eps_\mathrm{c}$ fixed} and updating only the perturbation 
matrices $U$ and $V$.
The algorithm repeats this expansion-contraction process in a loop until SVSA-RF  can 
no longer find a new perturbation that moves an eigenvalue off the imaginary axis into the right half-plane.
Following \cite{MitOve15}, the method is formally defined in Algorithm HEC-RF.
For an illustration of the main idea in the context of the complex stability radius, see \cite[Fig.~4.1]{MitOve15}.

\begin{algfloat}
\begin{algorithm}[H]
\floatname{algorithm}{}
\caption*{\textbf{Algorithm HEC-RF}: (Hybrid Expansion-Contraction: Real Frobenius-norm)}
\label{alg:rfhec}
\begin{algorithmic}[1]

	\algorithmicpurpose{
		to approximate $\stabradRf(A,B,C,D)$.
	}

	\REQUIRE{  
		$\eps_0 \in (0,\|D\|_2^{-1})$ and matrices $U \in \Rmn{p}{2}$ and $V \in \Rmn{m}{2}$ 
		with $\|UV\tp\|\fro=1$ and
      		$g_{UV}(\eps_0) > 0$, along with $\zlam_0$, 
      		a rightmost eigenvalue of $M_{UV}(\eps_0)$ in the right half-plane
	}
	\ENSURE{ 
		Final value of sequence $\{\eps_k\}$ such that $\zlam_k$ is a rightmost eigenvalue 
		of $M(\eps_k UV\tp)$ sufficiently close to the imaginary axis but in the closed right half-plane,
		certifying that $\eps_k \geq \stabradRf(A,B,C,D)$ 
		\\ 
		\quad \\
	}

	\FOR {$k = 0,1,2,\ldots$} 
		\STATE  \textbf{Contraction:} call a Newton-bisection zero-finding algorithm to compute \mbox{$\eps_\mathrm{c} \in (0,\eps_k]$} 
  so that $g_{UV}(\eps_\mathrm{c})$ $=0$, along with $\zlam_\mathrm{c}$, a rightmost eigenvalue of 
  $M_{UV}(\eps_\mathrm{c})$ on the imaginary axis.
  		\STATE  \textbf{Expansion:} call Algorithm SVSA-RF with input $\eps_\mathrm{c}$, $U$, $V$ 
   to compute $U_\mathrm{e}$, $V_\mathrm{e}$ with $\|U_\mathrm{e}V_\mathrm{e}\tp\|\fro=1$ and $\zlam_{k+1}$, a
   rightmost eigenvalue of $M(\eps U_\mathrm{e}V_\mathrm{e}\tp )$,
   satisfying $\real{\zlam_{k+1}} \ge \real{\zlam_\mathrm{c}} = 0$.
  		\STATE Set $\eps_{k+1} \coloneqq \eps_\mathrm{c}$, 
			$U \coloneqq U_\mathrm{e}$, and $V \coloneqq V_\mathrm{e}$.
	\ENDFOR

\end{algorithmic}
\end{algorithm}
\algnote{
In practice, we pass eigentriples computed by the contraction phase into the expansion phase and vice versa.  
}
\end{algfloat}

Convergence results for the original hybrid expansion-contraction algorithm developed for the complex stability radius were
given in \cite[Theorem 4.3]{MitOve15}. The basic convergence result that, under suitable assumptions, the sequence $\{\eps_k\}$ converges to some $\tilde\eps \geq \epsstar$ and the sequence $\{\real{\zlam_k}\}$ converges to zero, can be extended to the real Frobenius-norm case without difficulty.
However, the part that characterizes limit points of the sequence $\{\real{\zlam_k}\}$
as stationary points or local maxima of the norm of the transfer function on the stability boundary does not immediately 
extend to the real Frobenius-norm case, because instead of $\|G(\iu \omega)\|$, we would have to consider the 
potentially discontinuous function $\murf(G(\iu \omega))$. 

\subsection{The Derivatives of $g_{UV}$ and $g$} \label{subsec:derivs}

The contraction phase of the algorithm needs the derivative of $g_{UV}$ defined in \eqref{gdef} to implement the Newton-bisection
method to find a root of $g_{UV}$. As we shall see, it is also of
interest to relate this to the derivative of $g$ defined in \eqref{gUVdef}, although this is not actually used in
the algorithm. The key tool for obtaining both is Lemma \ref{lem:derE}, which presented the derivative of 
$\Delta(t) \left( I - D \Delta(t) \right)^{-1}$ with respect to $t$. Here, the same matrix function depends on $\eps$. 
We denote differentiation w.r.t.~$\eps$ by $'$.

\begin{theo} 
Let $O\subset(0,\|D\|_2^{-1})$ be open and suppose that, for all $\eps \in O$,
the rightmost eigenvalue $\zlam_{UV}(\eps)$ of $M_{UV}(\eps)$ in $\cuhp$ is simple and unique. 
Then, for all $\eps\in O$, $g_{UV}$ is differentiable at $\eps$ with
\begin{equation}\label{gUVprime_first}
        g_{UV}'(\eps)  =  \frac{\Real{ \left( y_{UV}(\eps)^*B U \right) 
        				\left[ I + \eps \, \left(V\tp DU\right) \left( I - \eps V\tp DU \right)^{-1}  \right]^2 
				\left( V\tp C x_{UV}(\eps) \right) }}{y_{UV}(\eps)^*x_{UV}(\eps)}
\end{equation}
where $(\zlam_{UV}(\eps)$, $x_{UV}(\eps)$, $y_{UV}(\eps))$ is a rightmost RP-compatible eigentriple of $M_{UV}(\eps)$.  
\end{theo}
\begin{proof}
Since $\zlam_{UV}(\eps)$ is simple and unique,
$x_{UV}(\eps)$ and $y_{UV}(\eps)$ are well defined (up to a unimodular scalar).
Applying Lemma \ref{lem:derE} with $\Delta(\eps)\equiv\eps UV\tp$, and using \eqref{eq:smw_first} -- \eqref{eq:smw_second}  and $\Xi \coloneqq I - \eps V\tp DU$,
we have
\begin{align}
        M_{UV}'(\eps)&  = B (I - \eps UV\tp D)^{-1}UV\tp (I-\eps D UV\tp)^{-1}C \label{firstMUVprime}\\
                               & = B (I + \eps U\Xi ^{-1}V\tp D) UV\tp (I + \eps DU\Xi^{-1}V\tp)C\nonumber\\
                               & = BU (I + \eps \Xi^{-1} V\tp DU)(I + \eps V\tp DU \Xi^{-1}) V\tp C \nonumber\\
                               & = BU \left[ I + \eps \left( V\tp DU\right) \left( I - \eps V\tp DU \right)^{-1}  \right]^2 V\tp C\nonumber
\end{align}
noting that $\Xi^{-1}$ and $V\tp D U$ commute.
Using standard eigenvalue perturbation theory, as in the proof of Lemma~\ref{lem:lambdaderiv}, we have
\begin{equation}
     g_{UV}'(\eps) = \Real{\zlam_{UV}'(\eps)} = \frac{\Real{y_{UV}(\eps)^*M_{UV}'(\eps)x_{UV}(\eps)}}{y_{UV}(\eps)^*x_{UV}(\eps)} \label{gUVprime}                
\end{equation}
from which the result follows.
\end{proof}

Now we obtain the derivative of the function $g$ defined in \eqref{gdef}.
\begin{theo} 
Let $O\subset(0,\|D\|_2^{-1})$ be open. Suppose that, for all $\eps \in O$,
\noindent \begin{enumerate}
	\item $\zlam(\eps)$ is the unique rightmost
point of $\SVSRepsf(A,B,C,D)$ (considering only those in $\cuhp$)
	\item $E(\eps)$, with $\|E(\eps)\|\fro=1$, is a smooth matrix function of $\eps$ such that $\zlam(\eps)$ is the unique rightmost eigenvalue of $M(\eps E(\eps))$ (again considering only those in $\cuhp$)
	\item $\Re(u(\eps)v(\eps)^*) \neq 0$ where 
			\begin{equation}
        u(\eps) = (I-\eps E(\eps)D)\tp[-] B\tp y(\eps), \quad v(\eps) = (I-\eps D E(\eps))^{-1} C\tp x(\eps),  \label{uvepsdef}
\end{equation}
and $(\zlam(\eps)$, $x(\eps)$, $y(\eps))$ is an RP-compatible eigentriple of $M(\eps E(\eps))$. 
\end{enumerate}  
~~~~~Then, for any $\eps \in O$,
\begin{equation}\label{gprime}
      g'(\eps) = \frac{\|\Real{u(\eps)v(\eps)^*}\|\fro}{y(\eps)^*x(\eps)}.
\end{equation}
\end{theo}
\begin{proof}
In this proof we again apply Lemma \ref{lem:derE} but with $\Delta(\eps)\equiv\eps E(\eps)$, obtaining
\begin{align*}
       M'(\eps) = & B (I - \eps E(\eps)D)^{-1}(E(\eps) + \eps E'(\eps)) (I-\eps D E(\eps))^{-1}C .
\end{align*}
Again using standard eigenvalue perturbation theory,  we have
\begin{align}
     g'(\eps) & = \Real{\zlam'(\eps)} = \frac{\Real{y(\eps)^*M'(\eps)x(\eps)}}{y(\eps)^*x(\eps)}  \nonumber \\
                  & = \frac{\Real{u(\eps)^*E(\eps)v(\eps)} + \eps\Real {u(\eps)^* E'(\eps) v(\eps)} } {y(\eps)^*x(\eps)} \nonumber \\
                  &=  \frac{\langle E(\eps) , \Re u(\eps)v(\eps)^*\rangle + \eps \langle E'(\eps) , \Re u(\eps)v(\eps)^*\rangle  } {y(\eps)^*x(\eps)}  \label{der-g-eps}     
\end{align}
The solution of the differential equation \eqref{ode-E} for $t\geq 0$ with initial condition $E(\eps)$ results in $\Re \dot \lambda (0) =0 $ as $\lambda(\eps)$ is a rightmost point. Therefore, by Theorem \ref{thm:ode}, as the case $\Re( u(\eps) v(\eps)^*) = 0$ is ruled out by the assumptions,  we have the identity 
   $$E(\eps) = \frac{\Real{ u(\eps) v(\eps)^*}}{\|\Real{ u(\eps) v(\eps)^*}\|\fro}. $$
Plugging this identity into \eqref{der-g-eps} and using the fact that 
$\langle E'(\eps), E(\eps) \rangle = \frac{1}{2} \frac{d \| E(\eps)\|_F^2}{d\eps} = 0$ we conclude that \eqref{gprime} holds.
\end{proof}

We now relate $g_{UV}'(\eps)$ to $g'(\eps)$.

\begin{theo} \label{derivs_same}
Using the notation established above, suppose
the assumptions of the two previous theorems apply for the same open interval $O$ and that for some 
specific $\eps\in O$,
\beq
        UV\tp = E(\eps) = \frac{\Real{u(\eps)v(\eps)^*}}{\|\Real{u(\eps)v(\eps)^*}\|\fro},   \label{UVdef}
\eeq
so that the matrices $M_{UV}(\eps)$ and $M(\eps E(\eps))$ are the same and the eigentriples 
$(\zlam_{UV}(\eps)$, $x_{UV}(\eps)$, $y_{UV}(\eps))$ and $(\zlam(\eps)$, $x(\eps)$, $y(\eps))$ coincide, with $g_{UV}(\eps)=g(\eps)$.
Then
\[
                 g_{UV}'(\eps)=g'(\eps).
\]
\end{theo}

\begin{proof}
Using  \eqref{gUVprime}, \eqref{firstMUVprime} and \eqref{UVdef} 
we have
\[
      g_{UV}'(\eps) = \frac{\Real{y(\eps)^* \left (B (I - \eps E(\eps)\tp D)^{-1}E(\eps) (I-\eps D E(\eps))^{-1}C\right )x(\eps)}}{y(\eps)^*x(\eps)}. 
\]
So, using \eqref{uvepsdef} and \eqref{gprime}, we obtain
\[
                            g_{UV}'(\eps)  = \frac{\Real{u(\eps)^* E(\eps) v(\eps)}}{y(\eps)^*x(\eps)} 
                            = \frac{\|\Real{u(\eps)v(\eps)^*}\|\fro}{y(\eps)^*x(\eps)} = g'(\eps).
\]
\end{proof}

This result is important, because at the start of the contraction phase of Algorithm HEC-RF, assuming that the expansion phase
has returned a locally rightmost point of $\SVSRepsf(A,B,C,D)$, we have that \eqref{UVdef} holds. 
Hence, the first Newton
step of the contraction phase, namely a Newton step for finding a zero of $g_{UV}$, is \emph{equivalent} to a Newton step for 
finding a zero of $g$, which is the ultimate goal. For this reason, under a suitable regularity condition, Algorithm HEC-RF is
actually quadratically convergent. We omit the details here, but a convergence rate analysis similar to that given in \cite[Theorem 4.4]{MitOve15}
for the complex stability radius hybrid expansion-contraction algorithm holds for Algorithm HEC-RF too. 

\section{Discrete-time systems}
\label{sec:discrete}
We now briefly summarize the changes to our results and algorithms that are needed to handle, instead of \eqref{eqn:AB}--\eqref{eqn:CD},
the discrete-time system
\begin{align*}
x_{k+1} & =  Ax_k + Bu_k \\
y_k & =  Cx_k + Du_k
\end{align*}
where $k=1,2,\ldots$. The definitions of the transfer matrix function, spectral value sets and real-$\mu$ functions in Section \ref{subsec:specvalsets}
remain unchanged. In Section \ref{subsec:stabrad}, the stability region is the open unit disk $\oud$ instead of the open left half-plane $\olhp$, 
and the definition of the stability radius must be adjusted accordingly. In Section \ref{subsec:specvalsetabsc}, 
instead of the spectral abscissa $\alpha$ and spectral value set abscissa $\aleps$, we require
the spectral radius $\rho$ and spectral value set radius $\rhoepsKn$, which are defined by maximization of $|\zlam|$ instead 
of $\real{\zlam}$ over the spectral value set.\footnote{Recall again the completely different usage of ``radius" in these names, the stability radius referring to the data space and the spectral
radius to the complex plane.}  Now, instead of ``rightmost" points, we search for ``outermost" points.

In Section \ref{sec:odeframework}, it is convenient to extend Definition \ref{def:eigentriple} as follows:
$(\lambda,x,y)$ is an RP($z$)-compatible eigentriple of $M$ if $\lambda$ is a simple eigenvalue of $M$,
$x$ and $y$ are corresponding normalized right and left eigenvectors and $y^*x$ is a real positive multiple of $z$.
Then, instead of \eqref{eq:optprob2}, we have, taking $(\lambda(t),x(t),y(t))$ to be an RP($\conjg{\lambda}(t)$)-compatible eigentriple,
\[ 
    \frac{d}{dt}|\lambda(t)| =\frac{\Re\conjg{\lambda}(t)\dot{\lambda}(t)}{|\lambda(t)|} = \frac{1}{|y(t)^*x(t)|}\Re(u(t)^* \dot{\Delta}(t) v(t)).
\]
The ODE \eqref{ode-E} then remains unchanged, except that the eigentriple $(\lambda(t),x(t),y(t))$ is an \emph{outermost} 
RP($\conjg{\lambda}(t)$)-compatible eigentriple of $M(\eps E(t))$ instead of a rightmost RP-compatible eigentriple.
Theorem \ref{thm:ode} also holds as before, with the same change. In Section \ref{sec:method}, we 
replace Algorithm SVSA-RF by Algorithm SVSR-RF (Spectral Value Set Radius: Real Frobenius-norm), whose purpose
is to approximate $\rhoepsRf$ $(A,B,C,D)$. The only change that is needed is to replace rightmost RP-compatible
eigentriple $(\lambda_k,x_k,y_k)$ by outermost RP($\conjg\lambda_k$)-compatible eigentriple $(\lambda_k,x_k,y_k)$.
To ensure that $|\zlam_{k+1}| \geq |\zlam_{k}|$, a line search can again be used,
as explained in \cite[Sec.~6.3.3]{Mit14}.  The algorithm produces $\zlam_k$ certifying that $|\zlam_k|\leq \rhoepsRf$ $(A,B,C,D)$.

In Section~\ref{sec:stabrad}, since the stability region is now the open unit disk $\oud$,
instead of \eqref{gdef} we have $g(\eps)=\rhoepsRf$ $(A,B,C,D) - 1$, and instead of \eqref{gUVdef}
we have $g_{UV}(\eps) \coloneqq \rho\left (M_{UV}(\eps) \right) - 1$. The derivatives of $g_{UV}$ in \eqref{gUVprime} 
and $g$ in \eqref{gprime} remain unchanged except for the RP-compatibility change and the replacement of $y^*x$ by $|y^*x|$ in
both denominators. Besides the RP-compatibility definition change, Algorithm HEC-RF is changed as follows:   
rightmost, right half-plane and imaginary axis are changed to outermost, $\C \backslash \D_-$ and unit circle respectively.

\section{Implementation and Experiments}
\label{sec:numerical}
We implemented Algorithm HEC-RF by extending the open-source \matlab\ code \texttt{getStabRadBound} \cite[Section 7]{MitOve15}, which is the implementation of the original HEC algorithm for approximating the complex stability radius.  Our new code supports approximating both the complex and the real Frobenius-norm bound stability radius, for both continuous-time and discrete-time systems, although for brevity, we continue to refer primarily only to the continuous-time case.
We similarly adapted the related fast upper bound algorithm \cite[Section 4.4]{MitOve15}, which aims to quickly find a destabilizing perturbation necessary for initializing Algorithm HEC-RF.  This ``greedy" strategy aims to take steps as large as possible towards a destabilizing perturbation by alternating between increasing $\eps$ and taking a single SVSA-RF update step 
of the perturbation matrices $U$ and $V$.  In the course of this work, we also significantly improved 
the convergence criteria of \texttt{getStabRadBound}.  As these issues are crucial for implementing a practicable and reliable version of the HEC algorithm, but the discussion does not specifically pertain to the real stability radius case, we defer the details to Appendix~\ref{sec:apdx_code}.  
 Lastly, we also made several improvements to help accelerate the algorithm.
  
\subsection{Acceleration features}
\label{sec:code_accel}

First, we extended \texttt{getStabRadBound}'s  feature for implicitly extrapolating the sequences of rank-1 perturbation matrices, produced in the expansion phases when approximating the complex stability radius, to also handle the corresponding rank-2 sequences that may occur when approximating the real stability radius.  To be efficient, the procedure takes a history of $U_k$ and $V_k$ matrices from Algorithm SVSA-RF and then forms four vector sequences corresponding to two selected rows and two selected columns of the evolving $U_kV_k\tp$ matrix.  Vector extrapolation is then applied individually to these four vector sequences to obtain two rows $r_1$ and $r_2$ and two columns $c_1$ and $c_2$ from the extrapolation of the sequence $\{U_kV_k\tp\}$, without ever explicitly forming these matrices.  Using the resulting four vector extrapolations, a new pair $U_\star$ and $V_\star$ are computed such that $\|U_\star V_\star\tp\|_F = 1$.  
For more details, see \cite[Section 6.3.5]{Mit14}.
 
Second, for SVSA-RF, we note that even if a full update step satisfies monotonicity, and thus does not require a line search, it may happen that the line search could still sometimes produce a better update anyway, particularly if \texttt{getStabRadBound}'s interpolating quadratic or cubic-based line search option is enabled; see \cite[Section 4.3]{Mit14}.  As such, even when a full SVSA-RF step satisfies monotonicity, our implementation will check whether an interpolated step might be even better, specifically, by considering the maximum of an interpolating quadratic line search model.  If this quadratic interpolation-derived step is predicted to be at least 1.5 times better than the already computed full update step, the rightmost eigenvalue of the corresponding interpolated perturbation is computed and if it satisfies monotonicity, with respect to the full step, then it is accepted in lieu of the full step. 

Finally, we modified the entire code to only compute left eigenvectors on demand, instead of always computing eigentriples.  If the resulting rightmost 
eigenvalue of $M(\Delta)$ for some perturbation $\Delta$ encountered by the algorithm does not satisfy monotonicity, then computing the corresponding left eigenvector is unnecessary; left eigenvectors need only be computed at points accepted by the algorithm, since it is only at such points that derivatives of eigenvalues are used.  This optimization essentially halves the cost of all incurred line searches in both Algorithm SVSA-RF and step 3 of the fast upper bound procedure, while it also halves the cost of computing extrapolation and interpolation steps that end up being rejected. 

\subsection{Numerical evaluation of Algorithm HEC-RF}
\label{sec:experiments}

\begin{table}
\small
\centering
\begin{tabular}{ l | cc | rr | rr | ll } 
\toprule 
\multicolumn{9}{c}{Small Dense Problems: Continuous-time}\\
\midrule
\multicolumn{1}{c}{} & \multicolumn{2}{c}{Iters} & \multicolumn{2}{c}{\# Eig} & \multicolumn{2}{c}{Time (secs)} & \multicolumn{2}{c}{RSR Approximation} \\
\cmidrule(lr){2-3}  
\cmidrule(lr){4-5}  
\cmidrule(lr){6-7}
\cmidrule(lr){8-9}
\multicolumn{1}{l}{Problem} & 
	\multicolumn{1}{c}{v1} & \multicolumn{1}{c}{v2} &	
	\multicolumn{1}{c}{v1} & \multicolumn{1}{c}{v2} &
	\multicolumn{1}{c}{v1} & \multicolumn{1}{c}{v2} &
	\multicolumn{1}{c}{$\min \{\eps_1,\eps_2\}$} & \multicolumn{1}{c}{$(\eps_1 - \eps_2)/\eps_1$} \\
\midrule

\texttt{CBM} & 3 & 3 & 122 & 79 & 8.646 & 5.577 & $4.46769464697 \times 10^{0}$ & $-8.5 \times 10^{-12}$\\
\texttt{CSE2} & 7 & 7 & 223 & 117 & 0.643 & 0.386 & $4.91783643704 \times 10^{1}$ & \multicolumn{1}{c}{-}\\
\texttt{CM1} & 2 & 3 & 91 & 81 & 0.198 & 0.203 & $1.22474487041 \times 10^{0}$ & \multicolumn{1}{c}{-}\\
\texttt{CM3} & 3 & 4 & 126 & 108 & 1.063 & 0.952 & $1.22290355805 \times 10^{0}$ & \multicolumn{1}{c}{-}\\
\texttt{CM4} & 3 & 4 & 222 & 181 & 8.181 & 6.680 & $6.30978638860 \times 10^{-1}$ & \multicolumn{1}{c}{-}\\
\texttt{HE6} & 11 & 9 & 20852 & 9972 & 13.828 & 8.305 & $2.02865555290 \times 10^{-3}$ & $+1.5 \times 10^{-10}$\\
\texttt{HE7} & 4 & 6 & 492 & 248 & 0.406 & 0.322 & $2.88575420548 \times 10^{-3}$ & $-3.2 \times 10^{-12}$\\
\texttt{ROC1} & 2 & 3 & 93 & 78 & 0.127 & 0.150 & $9.11416570667 \times 10^{-1}$ & $+2.3 \times 10^{-12}$\\
\texttt{ROC2} & 3 & 3 & 98 & 83 & 0.136 & 0.161 & $7.49812117968 \times 10^{0}$ & $+1.0 \times 10^{-10}$\\
\texttt{ROC3} & 4 & 4 & 204 & 117 & 0.211 & 0.209 & $7.68846259016 \times 10^{-5}$ & $-3.5 \times 10^{-11}$\\
\texttt{ROC4} & 1 & 1 & 40 & 40 & 0.084 & 0.134 & $3.47486815789 \times 10^{-3}$ & \multicolumn{1}{c}{-}\\
\texttt{ROC5} & 5.5 & 11 & 263 & 426 & 0.226 & 0.390 & $1.02041223979 \times 10^{2}$ & $-8.0 \times 10^{-9}$\\
\texttt{ROC6} & 4 & 4 & 149 & 80 & 0.174 & 0.182 & $3.88148973329 \times 10^{-2}$ & \multicolumn{1}{c}{-}\\
\texttt{ROC7} & 3 & 3 & 142 & 107 & 0.165 & 0.163 & $8.96564880558 \times 10^{-1}$ & \multicolumn{1}{c}{-}\\
\texttt{ROC8} & 3 & 4 & 160 & 114 & 0.183 & 0.194 & $2.08497314619 \times 10^{-1}$ & $+4.7 \times 10^{-7}$\\
\texttt{ROC9} & 5 & 8 & 235 & 173 & 0.223 & 0.297 & $4.20965764059 \times 10^{-1}$ & \multicolumn{1}{c}{-}\\
\texttt{ROC10} & 1 & 1 & 26 & 26 & 0.079 & 0.096 & $1.01878607021 \times 10^{1}$ & \multicolumn{1}{c}{-}\\

\bottomrule
\end{tabular}
\caption{
The ``Iters" columns show the number of HEC-RF iterations until termination for the ``v1" and ``v2" configurations of \texttt{getStabRadBound}; 
note that these can be fractional since the method may quit
after either a contraction or expansion phase.  The ``\# Eig" columns show the total number of eigensolves (the sum of the number of right and left eigenvectors computed) incurred while the ``Time (secs)" columns show the elapsed wall-clock time in seconds per problem for both code variants.  The left column under the ``RSR Approximation" heading shows the better (smaller) of the two real stability radius approximations
$\eps_1$ and $\eps_2$, respectively computed by ``v1" and ``v2" versions of the code.  The rightmost column show the relative difference between these two approximations, with positive values indicating that the ``v2" code produced a better  approximation.  Relative differences below the 
$10^{-12}$ optimality tolerances used for the code are not shown.
}
\label{table:dense_cont}
\end{table}

\begin{table}
\small
\centering
\begin{tabular}{ l | cc | rr | rr | ll } 
\toprule 
\multicolumn{9}{c}{Small Dense Problems: Discrete-time}\\
\midrule
\multicolumn{1}{c}{} & \multicolumn{2}{c}{Iters} & \multicolumn{2}{c}{\# Eig} & \multicolumn{2}{c}{Time (secs)} & \multicolumn{2}{c}{RSR Approximation} \\
\cmidrule(lr){2-3}  
\cmidrule(lr){4-5}  
\cmidrule(lr){6-7}
\cmidrule(lr){8-9}
\multicolumn{1}{l}{Problem} & 
	\multicolumn{1}{c}{v1} & \multicolumn{1}{c}{v2} &	
	\multicolumn{1}{c}{v1} & \multicolumn{1}{c}{v2} &
	\multicolumn{1}{c}{v1} & \multicolumn{1}{c}{v2} &
	\multicolumn{1}{c}{$\min \{\eps_1,\eps_2\}$} & \multicolumn{1}{c}{$(\eps_1 - \eps_2)/\eps_1$} \\
\midrule

\texttt{AC5} & 3 & 4 & 244 & 197 & 0.250 & 0.234 & $2.01380141605 \times 10^{-2}$ & $+6.3 \times 10^{-12}$\\
\texttt{AC12} & 2 & 2 & 35 & 51 & 0.098 & 0.131 & $9.33096040564 \times 10^{-2}$ & \multicolumn{1}{c}{-}\\
\texttt{AC15} & 5 & 6 & 143 & 94 & 0.178 & 0.179 & $4.22159665084 \times 10^{-2}$ & \multicolumn{1}{c}{-}\\
\texttt{AC16} & 4 & 5 & 119 & 78 & 0.157 & 0.181 & $7.75365184115 \times 10^{-2}$ & \multicolumn{1}{c}{-}\\
\texttt{AC17} & 5 & 5 & 222 & 150 & 0.201 & 0.219 & $3.35508111043 \times 10^{-6}$ & $+4.5 \times 10^{-10}$\\
\texttt{REA1} & 2 & 2 & 77 & 66 & 0.116 & 0.167 & $1.37498793652 \times 10^{-3}$ & \multicolumn{1}{c}{-}\\
\texttt{AC1} & 4 & 5 & 325 & 230 & 0.267 & 0.253 & $7.99003318082 \times 10^{0}$ & \multicolumn{1}{c}{-}\\
\texttt{AC2} & 3 & 4 & 61 & 57 & 0.111 & 0.171 & $3.36705685350 \times 10^{0}$ & \multicolumn{1}{c}{-}\\
\texttt{AC3} & 4 & 4 & 427 & 297 & 0.305 & 0.329 & $7.43718998002 \times 10^{-2}$ & \multicolumn{1}{c}{-}\\
\texttt{AC6} & 5 & 9.5 & 253 & 366 & 0.215 & 0.356 & $2.32030683553 \times 10^{-8}$ & $+2.6 \times 10^{-1}$\\
\texttt{AC11} & 5 & 3 & 198 & 113 & 0.213 & 0.167 & $5.21908412146 \times 10^{-8}$ & $-2.5 \times 10^{-8}$\\
\texttt{ROC3} & 4 & 5 & 204 & 187 & 0.209 & 0.264 & $5.30806020326 \times 10^{-2}$ & \multicolumn{1}{c}{-}\\
\texttt{ROC5} & 6 & 5 & 280 & 176 & 0.246 & 0.306 & $2.85628817204 \times 10^{-4}$ & $-1.7 \times 10^{-10}$\\
\texttt{ROC6} & 5 & 7 & 324 & 111 & 0.269 & 0.239 & $5.81391974240 \times 10^{-2}$ & $+1.8 \times 10^{-1}$\\
\texttt{ROC7} & 4 & 4 & 68 & 55 & 0.115 & 0.161 & $9.01354011348 \times 10^{-1}$ & \multicolumn{1}{c}{-}\\
\texttt{ROC8} & 3 & 6 & 134 & 119 & 0.163 & 0.217 & $2.08192687301 \times 10^{-5}$ & $+1.6 \times 10^{-10}$\\
\texttt{ROC9} & 3 & 4 & 137 & 101 & 0.160 & 0.177 & $4.07812890254 \times 10^{-2}$ & \multicolumn{1}{c}{-}\\
\bottomrule
\end{tabular}
\caption{
See caption of Table~\ref{table:dense_cont} for the description of the columns.
}
\label{table:dense_disc}
\end{table}

\begin{table}
\small
\centering
\begin{tabular}{ l | cc | rr | rr | ll } 
\toprule 
\multicolumn{9}{c}{Large Sparse Problems: Continuous-time (top), Discrete-time (bottom)}\\
\midrule
\multicolumn{1}{c}{} & \multicolumn{2}{c}{Iters} & \multicolumn{2}{c}{\# Eig} & \multicolumn{2}{c}{Time (secs)} & \multicolumn{2}{c}{RSR Approximation} \\
\cmidrule(lr){2-3}  
\cmidrule(lr){4-5}  
\cmidrule(lr){6-7}
\cmidrule(lr){8-9}
\multicolumn{1}{l}{Problem} & 
	\multicolumn{1}{c}{v1} & \multicolumn{1}{c}{v2} &	
	\multicolumn{1}{c}{v1} & \multicolumn{1}{c}{v2} &
	\multicolumn{1}{c}{v1} & \multicolumn{1}{c}{v2} &
	\multicolumn{1}{c}{$\min \{\eps_1,\eps_2\}$} & \multicolumn{1}{c}{$(\eps_1 - \eps_2)/\eps_1$} \\
\midrule

\texttt{NN18} & 1 & 2 & 27 & 37 & 1.833 & 2.430 & $9.77424680376 \times 10^{-1}$ & \multicolumn{1}{c}{-}\\
\texttt{dwave} & 2 & 4 & 72 & 59 & 32.484 & 28.794 & $2.63019715625 \times 10^{-5}$ & \multicolumn{1}{c}{-}\\
\texttt{markov} & 2 & 3 & 61 & 56 & 12.581 & 10.479 & $1.61146532880 \times 10^{-4}$ & \multicolumn{1}{c}{-}\\
\texttt{pde} & 4 & 5 & 128 & 79 & 4.670 & 3.011 & $2.71186478815 \times 10^{-3}$ & \multicolumn{1}{c}{-}\\
\texttt{rdbrusselator} & 2 & 3 & 50 & 45 & 4.517 & 4.402 & $5.47132014748 \times 10^{-4}$ & \multicolumn{1}{c}{-}\\
\texttt{skewlap3d} & 2 & 2 & 103 & 78 & 115.004 & 90.605 & $4.59992022215 \times 10^{-3}$ & \multicolumn{1}{c}{-}\\
\texttt{sparserandom} & 2 & 2 & 90 & 74 & 3.056 & 2.655 & $7.04698184529 \times 10^{-6}$ & $-3.5 \times 10^{-10}$\\

\midrule

\texttt{dwave} & 2 & 4 & 34 & 33 & 14.776 & 14.397 & $2.56235064981 \times 10^{-5}$ & \multicolumn{1}{c}{-}\\
\texttt{markov} & 3 & 3 & 73 & 64 & 15.740 & 14.056 & $2.43146945130 \times 10^{-4}$ & \multicolumn{1}{c}{-}\\
\texttt{pde} & 2 & 2 & 46 & 35 & 1.713 & 1.450 & $2.77295935785 \times 10^{-4}$ & \multicolumn{1}{c}{-}\\
\texttt{rdbrusselator} & 3 & 5 & 76 & 60 & 5.926 & 5.041 & $2.56948942080 \times 10^{-4}$ & \multicolumn{1}{c}{-}\\
\texttt{skewlap3d} & 2 & 3 & 50 & 52 & 53.297 & 50.526 & $3.40623440406 \times 10^{-5}$ & \multicolumn{1}{c}{-}\\
\texttt{sparserandom} & 2 & 2 & 21 & 19 & 1.035 & 0.880 & $2.53298721605 \times 10^{-7}$ & $-3.5 \times 10^{-8}$\\
\texttt{tolosa} & 3 & 3 & 98 & 53 & 10.097 & 6.251 & $2.14966549184 \times 10^{-7}$ & $-6.5 \times 10^{-12}$\\

\bottomrule
\end{tabular}
\caption{
See caption of Table~\ref{table:dense_cont} for the description of the columns.
}
\label{table:sparse}
\end{table}

We tested our new version of \texttt{getStabRadBound} on the 34 small-scale and 14 large-scale linear dynamical systems 
used in the numerical experiments of \cite{GuGuOv13} and \cite{MitOve15}, noting that the system matrices $(A,B,C,D)$ for these problems are all real-valued.  We ran the code in two different configurations for each problem: once in its ``pure" HEC form, which we call ``v1" and which should converge quadratically, and a second time using an ``accelerated" configuration, which we call ``v2".  This latter configuration enabled both of the interpolation and extrapolation features described above as well as the code's early contraction and expansion termination conditions, both of which aim to encourage the code to accept inexact and cheaply-acquired solutions to the subproblems when more accuracy is neither needed nor useful.  The early expansion termination feature is described in \cite[Section 4.3]{MitOve15} and often greatly lessens the overall computational cost, despite the fact that it reduces the algorithm's theoretical convergence rate from quadratic to superlinear.  We used $0.01$ as the relative tolerances for governing these early contraction/expansion features.  The expansion and contraction optimality tolerances were both set to $10^{-12}$; these two tolerances act together to determine an overall tolerance for the HEC iteration.  The ``v2" configuration was set to attempt  extrapolation every fifth iteration, from the previous five iterates.  For all other parameters, we used \texttt{getStabRadBound}'s default user options.  All experiments were performed using \matlab\ R2015a running on a Macbook Pro with an Intel i7-5557U dual-core CPU and 16GB of RAM, running Mac OS X v10.11.5.

\begin{figure}[!h]
\begin{tabular}{@{} c @{} c @{}}
\includegraphics[scale=.41]{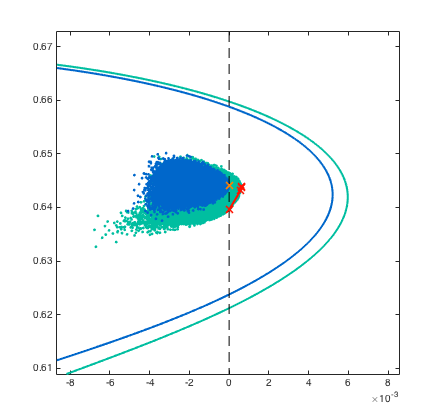} &
\includegraphics[scale=.41]{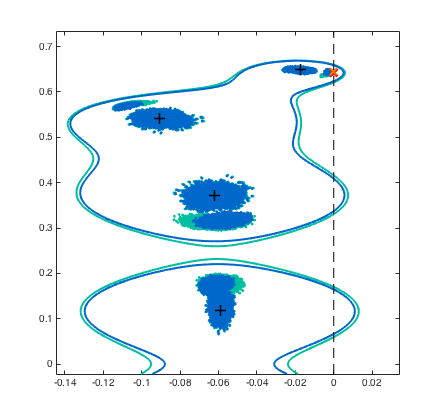} 
\end{tabular}
\caption{
Test problem \texttt{ROC3}.  
Selected iterates of Algorithm HEC-RF, namely the first and last expansion phases 
($\eps_1$ and $\eps_{4}$), are depicted as two sequences of x's connected by line segments,
respectively in red and orange, in a close-up view (left) and in a wide view (right).
The corresponding sets $\SVSRepsf(A,B,C,D)$  
were realized by plotting points of $\spec(M(\eps U_rV_r\tp))$, in green for $\eps=\eps_1$ and in blue for $\eps=\eps_4$, using many rank-1 and rank-2 ``sample" matrices $U_rV_r\tp$ with $\|U_rV_r\tp\|_F = 1$.  
Specifically, for each value of $\eps$, we used 100,000 randomly generated matrices $U_rV_r\tp$ (using \texttt{randn()}), another 100,000 generated via quasi-random Sobol sequences, and 100,000 randomly perturbed versions of the expansion phases' sequences of matrices.  The 200,000 random and quasi-random samples were unable to capture the region near the locally rightmost point found by Algorithm HEC-RF; the 
points from these samples only appear in the wider view on the right, in small regions about the eigenvalues of $A$ (represented by the black +'s). 
The sample points shown in the close-up view on the left are all from the randomly perturbed versions of the expansion phases' matrix iterates, demonstrating Algorithm HEC-RF ability to efficiently find extremal rightmost values in real-valued spectral value sets.  The solid curves depict the boundaries of the corresponding sets $\SVSCepss(A,B,C,D)$ and were computed by \matlab's \texttt{contour}.  As can be readily seen, 
the iterates of Algorithm HEC-RF converged to a locally rightmost of 
$\SVSRepsf(A,B,C,D)$, $\eps=\eps_4$, close to the imaginary axis (represented by the dashed vertical line) and in the interior of $\SVSCepss(A,B,C,D)$.
}
\label{fig:example_roc3}
\end{figure}

In order to assess whether Algorithm SVSA-RF converges to locally rightmost points, we have relied upon plotting approximations of $\SVSRepsf(A,B,C,D)$ in the complex plane using various random sampling and perturbation techniques (in contrast to the case of the complex stability radius, where the boundaries of $\SVSCepss(A,B,C,D)$ can be plotted easily).  For each of the 34 small-scale problems, we plotted the iterates of the expansion phases along with our $\SVSRepsf(A,B,C,D)$ approximations for the corresponding values of $\eps$ and examined them all by hand, observing that Algorithm SVSA-RF indeed does converge to locally rightmost points as intended, at least to the precision that can be assessed from such plots.  See Figures~\ref{fig:example_roc3} and \ref{fig:example_roc9} for two such plots.

\begin{figure}[!h]
\begin{tabular}{@{} c @{} c @{}}
\includegraphics[scale=.41]{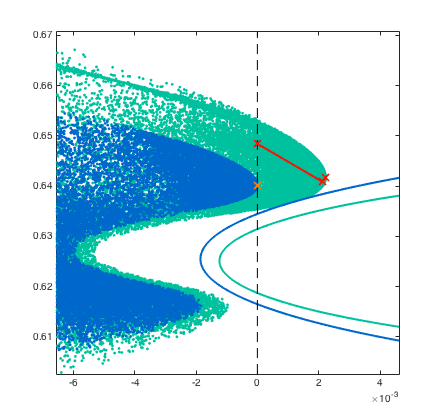} &
\includegraphics[scale=.41]{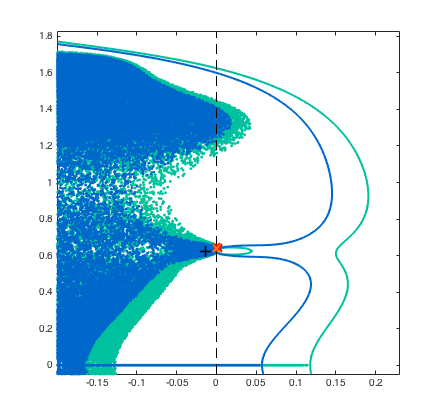} 
\end{tabular}
\caption{
Test problem \texttt{ROC9}.  
Selected iterates of Algorithm HEC-RF, namely the first and last expansion phases 
($\eps_1$ and $\eps_{5}$), are depicted as two sequences of x's connected by line segments,
respectively in red and orange, in a close-up view (left) and in a wide view (right).
The corresponding sets of point clouds for $\SVSRepsf(A,B,C,D)$ and set boundaries of $\SVSCepss(A,B,C,D)$,  in green for $\eps=\eps_1$ and in blue for $\eps=\eps_5$, were plotted in a similar manner as described in Figure~\ref{fig:example_roc3}.  The black + represents an eigenvalue of $A$.  
As can be seen, the iterates of Algorithm HEC-RF converged to a locally rightmost point of 
$\SVSRepsf(A,B,C,D)$, with $\eps=\eps_5$, close to the imaginary axis (represented by the dashed vertical line),
though in this case, it is clear that this is not a globally rightmost point.
Interestingly, the corresponding sets $\SVSCepss(A,B,C,D)$ have 
no locally rightmost points near the sequence of locally rightmost points of $\SVSRepsf(A,B,C,D)$ found by Algorithm HEC-RF, highlighting the striking difference between real-valued and complex-valued spectral value sets.  In fact,
for $\eps = \eps_1$, it is seen that $\SVSCepss(A,B,C,D)$ actually has a hole below and to the right of the locally rightmost point of $\SVSRepsf(A,B,C,D)$ found by Algorithm HEC-RF; the hole is depicted by the small green ellipse in the right plot, a portion of which can be seen in the left plot.
}
\label{fig:example_roc9}
\end{figure}

 As Algorithm HEC-RF is, to the best of our knowledge, the only available method to approximate the real Frobenius-norm bounded stability radius, we simply report the resulting upper bounds produced by our method to 12 digits for each test problem, along with statistics on the computational cost, in Tables~\ref{table:dense_cont}-\ref{table:sparse}.  We observe that both variants of the code tend to produce 
 approximations with high agreement, showing that there seems to be little to no numerical penalty for enabling the acceleration features.  In fact, on two  examples (\texttt{AC6} and \texttt{ROC6}, both discrete-time systems), we see that the accelerated version of the code actually produced substantially better approximations, with improvement to their respective second-most significant digits.  Furthermore, the accelerated ``v2" 
 configuration does appear to be effective in reducing the number of eigensolves incurred on most problems, though there are two notable exceptions to 
 this: \texttt{ROC5} (continuous-time) and \texttt{ROC6} (discrete-time).  It is worth noting that many of the small-scale test problems have such tiny dimensions that a reduction in eigensolves doesn't always correspond with a speedup in terms of wall-clock time (and can sometimes seemingly paradoxically have increased running times due to the inherent variability in collecting timing data).  However, on only moderate-sized problems, such as \texttt{CBM}, \texttt{CSE2}, and \texttt{CM4} (all continuous-time), we start to see the correspondence between number of eigensolves and running time approaching a one-to-one relationship.  This correspondence is readily apparent in the large and sparse examples in Table~\ref{table:sparse}.
 
Though it is difficult to tease out the effects of the different acceleration options, since they interact with each other, we were able to determine that the early expansion  termination feature was usually the dominant factor in reducing the number of eigensolves.  However, extrapolation was crucial for the large gains observed on \texttt{HE6} (continuous-time) and \texttt{ROC6} (discrete-time).   By comparison, in \cite{MitOve15}, when using HEC to approximate the complex stability radius, extrapolation tended to be much more frequently beneficial while usually providing greater gains as well.  Part of this disparity may be because of the greatly increased number of eigensolves we observed when running \texttt{getStabRoundBound} to approximate the complex stability radius as opposed to the real stability radius; on the 34 small-scale problems, the complex stability radius variant incurred 1226 more eigensolves per problem on average, with the median being 233 more.  In our real stability radius experiments, \texttt{HE6} notwithstanding, Algorithm SVSA-RF simply did not seem to incur slow convergence as often nor to the same severity as its rank-1 counterpart for complex-valued spectral value sets.  We note that our ODE-based approach for updating real rank-2 Frobenius-norm bounded perturbations  underlying Algorithm SVSA-RF also provides a new expansion iteration for complex spectral value sets; in Appendix~\ref{sec:complex_svsa}, we evaluate the performance of this new variant when approximating the complex stability radius.

\subsection{New challenges for the real stability radius case}

Over the test set, only a handful of examples triggered our new rank-2 extrapolation routine: continuous-time problems \texttt{ROC1} and \texttt{ROC3} and
discrete-time problems \texttt{AC5}, \texttt{AC1}, \texttt{AC3}, \texttt{ROC3}, and \texttt{ROC5}.  Of these seven, the code only produced a successful extrapolation 
for \texttt{ROC1}, which is seemingly not a promising result for the rank-2 extrapolation procedure.   However, perhaps none of these problems ended up being particularly good candidates for evaluating the rank-2 extrapolation procedure; their respective total number of perturbation updates to $U_kV_k\tp$, with all acceleration features disabled, was at most 160 updates (\texttt{AC3}), with the average only being 70.7.  Simply put, these problems provided little opportunity for any extrapolation, yet alone need.

As an alternative to recovering the aforementioned normalized $U_\star$ and $V_\star$ matrices by the direct procedure described in Section~\ref{sec:code_accel},
we also considered specifying it as a constrained optimization problem.  For notational convenience, we assume that $r_1$, $r_2$ and $c_1$, $c_2$ are respectively the first and second rows and columns of the implicitly extrapolated matrix of $\{U_kV_k\tp\}$.  To recover $U_\star$ and $V_\star$ from the extrapolated rows $r_1$, $r_2$ and columns $c_1$, $c_2$, we instead solve the following constrained optimization problem
\beq
\label{eq:extrap_opt}
\min_{\|UV\tp\|_F = 1} 
\left\| 
\begin{aligned} 
\begin{bmatrix} c_1 & c_2 \end{bmatrix} & - U V(1:2,:)\tp \\ 
\begin{bmatrix} r_1(3:\texttt{end})\\ r_2(3:\texttt{end}) \end{bmatrix}  & - U(1:2,:)V(3:\texttt{end},:)\tp
\end{aligned}								
\right\|_2.
\eeq
We derived the gradients of the above objective function and equality constraint and used them with both \matlab's \texttt{fmincon} and \granso: GRadient-based Algorithm for Non-Smooth Optimization \cite{CurMitOve16}, to solve
\eqref{eq:extrap_opt}.  We observed that the vectors obtained by our direct procedure frequently made excellent starting points for solving \eqref{eq:extrap_opt}, often greatly reducing the iteration numbers incurred by \texttt{fmincon} and \granso\ compared to initializing these codes from randomly generated starting points.  Regardless of the starting points employed, the optimization routines typically only found solutions that reduced the objective function of \eqref{eq:extrap_opt} by at most an order of magnitude compared to the solutions obtained by the direct procedure, and the resulting extrapolations were generally no better in terms of acceptance rate than the ones produced by our direct procedure.  This seems to either confirm that these particular problems are poorly suited for evaluating extrapolation, as suggested above, or indicate that it is perhaps the quality, or lack thereof, of the vector extrapolations themselves, rows $r_1$, $r_2$ and columns $c_1$, $c_2$, that is causing the difficulties in producing good rank-2 extrapolations.   We examined the vector sequences used to create $r_1$, $r_2$ and $c_1$, $c_2$ and did see significant oscillation, which appears to be a property of the iterates of Algorithm SVSA-RF itself, at least for these particular problems; this oscillation may be an additional difficulty to overcome for improved rank-2 extrapolation performance, but we leave such investigation for future work.

\begin{figure}[!h]
\centering
\begin{tabular}{@{} c @{} c @{}}
\includegraphics[scale=.36]{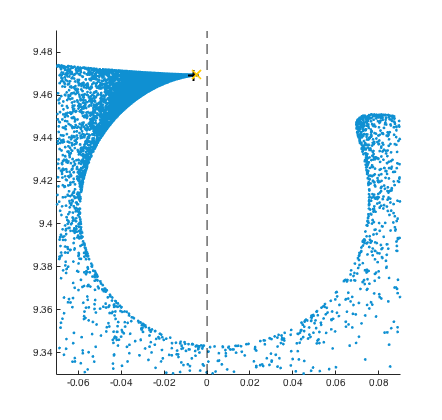} &
\includegraphics[scale=.36]{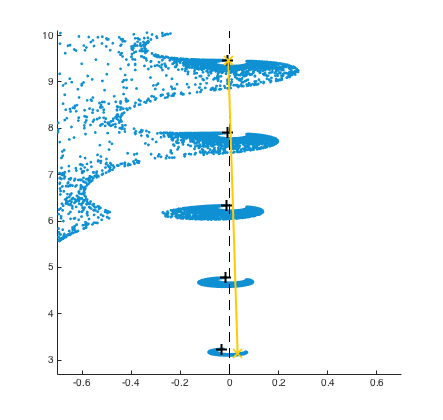}\\
\includegraphics[scale=.36]{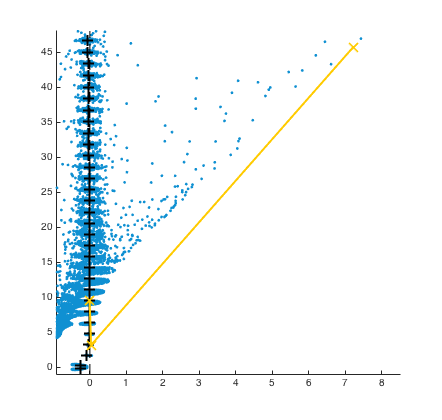} &
\includegraphics[scale=.36]{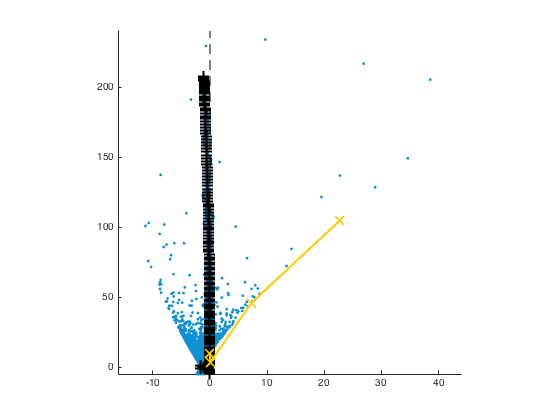}\\
\end{tabular}
\caption{
Test problem \texttt{CM4}.
Successively wider views (left-right, top-to-bottom) of $\SVSRepsf(A,B,C,D)$ (realized by the blue dot samples and generated in a similar manner as described in Figure~\ref{fig:example_roc3}) showing selected iterates of the upper bound procedure (yellow x's connected 
by line segments)  with $\eps$ near its limit of $\|D\|^{-1}$.  The black +'s are eigenvalues of $A$ while the black dashed line is the imaginary axis.
Top left: the expansion phase of the upper bound procedure has nearly converged to a locally rightmost point that is 
just to the right of an eigenvalue of $A$ but this region of $\SVSRepsf(A,B,C,D)$ is always contained in the left half-plane 
due to the limit $\eps < \|D\|^{-1}$.  Top right: this highly nonconvex ``horseshoe" structure is repeated in multiple places; on the next expansion step, the routine was able to ``jump" to a different region of $\SVSRepsf(A,B,C,D)$ that is in the right half-plane.
Bottom left: on the next step after that, the expansion phase again jumps out of one region to another,
this time significantly farther to the right; though technically an upper bound had already been found, the routine continued to iterate to better locate where a minimal destabilizing perturbation may lie.  Bottom right: zooming out further, we see
that $\SVSRepsf(A,B,C,D)$ is much larger than was initially apparent.
}
\label{fig:ub_traj}
\end{figure}

Lastly, we noticed that the fast upper bound procedure had some difficulty before it was able to find a destabilizing perturbation for problem \texttt{CM4} (continuous-time).  
Generally, we have found that the upper bound procedure can find a destabilizing perturbation within a handful of iterations, but on \texttt{CM4}, it
took 23 iterations.  In Figure~\ref{fig:ub_traj}, we show plots of $\SVSRepsf(A,B,C,D)$ for the largest value of $\eps$ obtained in the upper bound procedure, along with selected iterates of the routine corresponding to that value of $\eps$.  As is apparent from the plots, part of the difficulty in finding an upper bound is due to the highly nonconvex ``horseshoe" shapes that create locally rightmost points in the left half-plane for 
values of $\eps$ near its upper bound $\|D\|^{-1}$.
The expansion routine had converged to such a point and then iteratively increased epsilon to be near its upper bound in vain.  However, and surprisingly, on the 23rd iteration of the upper bound procedure, the expansion phase was actually able to jump out of this
region and land in the right half-plane to find a destabilizing perturbation and thus an upper bound.  Even though the routine had essentially already converged to a perturbation corresponding to this locally rightmost point in the left half-plane, the routine still produced another update step to try, but this update step was nearly identical to the current perturbation, because further rightward continuous progress was not possible.  The full update step failed to satisfy monotonicity so the line search was invoked with an initial interpolation of $t=0.5$ and as a result, the resulting unnormalized interpolation of the current perturbation and the full update step perturbation nearly annihilated each other.   When that interpolation was renormalized back to have unit Frobenius norm, as is necessary, the resulting perturbation was very different than the current perturbation (as well as the full update step), which thus allowed the algorithm to jump to an entirely different disconnected region of the spectral value set.  We note that Algorithm SVSA-RF can also ``jump" when a new perturbation just happens to result in a second eigenvalue of $A$ being taken farther to the right than the eigenvalue it had intended to push rightward.

\section{Conclusion}
\label{sec:conclusion}
We have presented an algorithm that, to our knowledge, is the first method available to approximate the real stability radius of a linear dynamical system with inputs and outputs defined using Frobenius-norm bounded perturbations. It is efficient even in the large scale case, and since it generates destabilizing perturbations explicitly, it produces guaranteed upper bounds on the real stability radius.  The hybrid expansion-contraction method works by alternating between (a) iterating over a sequence of destabilizing perturbations of fixed norm $\eps$ to push an eigenvalue of the corresponding perturbed system matrix as far to the right in the complex plane as possible and (b) contracting $\eps$ to bring the rightmost eigenvalue back to the imaginary axis. The final computed eigenvalue is very close to the imaginary axis and is typically at least a locally rightmost point of the corresponding $\eps$-spectral value set. The method is supported by our theoretical results for the underlying ODE that motivates the method, and our computational results are validated by extensive random sampling techniques. The method has been implemented in our open-source \matlab\ code \texttt{getStabRadBound}.



\medskip
\bibliographystyle{alpha}
\bibliography{rsr_refs}  

\begin{appendices}
\section{Notes on implementing HEC}
\label{sec:apdx_code}
We use HEC here to refer to the original HEC algorithm \cite{MitOve15} and Algorithm HEC-RF since the 
issues outlined in this section apply equal to both methods.
In theory, HEC converges once a locally rightmost point $\lambda$ with $\Real{\lambda} = 0$ has been 
found.    In practice however, 
it is not so straightforward.  Indeed, we have discovered that the HEC convergence criteria described in 
\cite[Section 7.1]{MitOve15}, using the expansion phase stopping condition proposed in \cite[Section 5]{GuGuOv13},
can sometimes be inadequate.  We briefly recap these conditions and then present improved criteria, which 
we have used in our new version of \texttt{getStabRadBound} and all experiments in this paper.

To ensure that iterates always remain in the right half-plane (which is necessary for provable convergence of HEC), 
the Newton-bisection based contraction phase is set to find a point on the line $x=\tau_\eps/2$, 
where $\tau_\eps > 0$ is the contraction tolerance.
The contraction is said to have converged if it finds a point $\lambda_c \in [0,\tau_\eps)$.  This 
shift permits either right or left-sided convergence;
for the unshifted problem, i.e. $x=0$, convergence from the left would almost always fail to satisfy the
requirement that HEC points remain in the right half-plane.
 The expansion phase then pushes rightward from $\lambda_0 \coloneqq \lambda_c$ with iterates $\lambda_k$, where $\real{\lambda_{k+1}} > \real{\lambda_k}$
 holds for all $k \ge 0$ (due to the line search ensuring monotonicity).
 In \cite{GuGuOv13}, it was proposed that the expansion phase should be halted once
$\real{\lambda_{k+1}} - \real{\lambda_k} < \tau_{uv} \max(1, \real{\lambda_{k}})$ is satisfied.\footnote{Note that 
compared to  \cite{GuGuOv13}, we have dropped 
the absolute value signs here, since in the context of HEC, all points lie in the right half-plane.}
In \cite{MitOve15}, HEC was then said to have converged if either (a) the expansion phase returned a point $\lambda$ such 
that $\real{\lambda} < \tau_\eps + \tau_{uv}$, where $\tau_{uv}  > 0$ is the expansion tolerance, or (b) if the 
expansion and contraction phases failed consecutively, in either order.  We now
discuss the inadequacies and propose improved conditions.

First, we make the change that HEC is said to have converged if $\real{\lambda} \in [0,2\tau_\eps)$,
where $\lambda$ is a locally rightmost point of the associated spectral value set encountered by HEC.  
Unfortunately, detecting whether $\lambda$ is locally rightmost is somewhat problematic due
to the lack of an efficiently computable optimality measure.
Furthermore, the expansion phase stopping condition described above may cause expansion to halt prematurely, 
which in turn can cause HEC to return an unnecessarily higher value of $\eps$, that is, a worse approximation.  
Part of this is due to the fact that when the expansion iterates have real part less than one, the condition only 
measures an absolute difference between consecutive steps, which is often a rather poor optimality measure.  
Furthermore, by only measuring
the difference between the real parts, it fails to capture change in the imaginary parts, which if present, would strongly indicate
that a locally rightmost point has not yet been reached.  To address both concerns, we instead propose to measure 
the relative change \emph{in the complex plane} between consecutive iterates, that is, 
$| \lambda_{k+1} - \lambda_k | / | \lambda_k |$, and to halt the expansion phase when this value falls below $\tau_{uv}$.  
Doing so however creates two additional wrinkles that must now be addressed.  The condition is not meaningful
when either point is exactly zero.  Therefore, we simply skip this particular check if either is indeed zero, noting that it can, at
most, only add two additional steps to be taken (provided no other termination condition is triggered).
Also, we must prevent ``false" oscillation in imaginary parts from being measured due to complex conjugacy; 
when matrices $A$, $B$, $C$, and $D$ are all real, we simply ensure that distance is
measured between eigenvalues in the upper half-plane only, by flipping signs of the imaginary parts as necessary.

\begin{table}
\small
\centering
\begin{tabular}{c p{11.8cm}} 
\toprule 
\multicolumn{2}{c}{Simplified list of expansion/contraction termination possibilities}\\
\midrule
\multicolumn{2}{l}{Contraction Phase:}\\
\texttt{0:} & \texttt{maxit reached though some contraction achieved}\\
\texttt{1:} & \texttt{point with real part in $[0,\tau_\eps)$ attained}\\
\texttt{2:} & \texttt{significant contraction achieved in right half-plane, halted early}\\
\texttt{3:} & \texttt{desired point is bracketed by two consecutive floating point numbers (no convergence criteria satisfied)}\\
\midrule
\multicolumn{2}{l}{Expansion Phase:}\\
\texttt{0:} & \texttt{maxit reached}\\
\texttt{1:} & \texttt{relative difference, in $\C$, between $\lambda_{k+1}$ and $\lambda_k$ is less than $\tau_{uv}$}\\
\texttt{2:} & \texttt{step length $\real{\lambda_{k+1}} - \real{\lambda_k}$ has significantly shortened, halted early}\\
\texttt{3:} & \texttt{line search failed to produce a monotonic step}\\
\bottomrule
\end{tabular}
\caption{For both phases, return code 1 is used to indicate that the respective desired convergence
has been achieved, while return code 2 indicates that the early termination features were invoked.  For the contraction
phase, return code 3 indicates that the precision limits of the hardware precludes satisfying the convergence criteria, i.e.,
it is a hard failure.  For the expansion phase, return code 3 can be interpreted as a sign of convergence, as 
further rightward progress is apparently no longer possible.
}
\label{table:return_codes}
\end{table}

Another major issue is that, in practice, both the contraction and expansion phases may terminate in a multitude of ways 
without satisfying their respective convergence criteria, including our improved expansion-halting condition above.  
In Table~\ref{table:return_codes}, we give a simplified list of these termination possibilities.  We now describe how
to interpret the combination of these possibilities that can occur and how their consequences should be handled and
present a concise pseudocode for the resulting practical implementation of HEC.

\begin{algfloat}
\begin{algorithm}[H]
\floatname{algorithm}{}
\caption*{\textbf{Pertinent pseudocode for a practical HEC implementation}}
\label{alg:hec_practical}
\begin{algorithmic}[1]
	\REQUIRE{  
		$\eps_0 > 0$, rightmost $\lambda_0 \in \sigma(M(\eps_0 UV^*))$ such that $\real{\lambda_0} > 0$
		with $\|UV^*\| = 1$ for the chosen norm and $U \in \C^p$, $V \in \C^m$ or  
		$U \in \Rmn{p}{2}$, $V \in \Rmn{m}{2}$ for respectively
		approximating the complex or real stability radius radius,
		initial contraction bracket given by $\epslb \coloneqq 0$
		 and $\epsub \coloneqq \eps_0$, boolean variable \texttt{expand\_converged $\coloneqq$ false}, $\ldots$
	}\\
	\ENSURE{ 
		Final value of sequence $\{\eps_k\}$.
	}
	\quad \\ \quad \\
	
	\FOR {$k = 0,1,2,\ldots$} 

		\STATE $[\eps_\mathrm{c}, \lambda_\mathrm{c}, \epslb, \epsub, \texttt{ret\_con}] =$ \texttt{contract($\epslb,\epsub,\lambda_k,\ldots$)}
		\COMMENT{where:}
		\STATE \COMMENT{\qquad $0 < \hat\eps \le \eps$ is the possibly contracted value of $\eps$}
		\STATE \COMMENT{\qquad  $0 \le \real{\lambda_\mathrm{c}} \le \real{\lambda_k}$ is the possibly contracted eigenvalue}
		\STATE \COMMENT{\qquad possibly updated $\epslb$ and $\epsub$ giving the tightest bracket encountered}
		\STATE \COMMENT{\qquad \emph{\texttt{ret\_con}} is the contraction's return code from Table~\ref{table:return_codes}}
		\STATE \COMMENT{Check if no contraction was possible (precision of hardware exhausted)}
		\IF{\texttt{ret\_con} $==3$ \AND \texttt{expand\_converged} }
			\IF{$\real{\lambda_c} < 2 \cdot \tau_\eps$}
				\RETURN \COMMENT{HEC converged to tolerance}
			\ELSE
				\RETURN \COMMENT{HEC stagnated}
			\ENDIF
		\ENDIF	
		\STATE $\eps_{k+1} \coloneqq \eps_\mathrm{c}$ 

		\STATE $[\lambda_{k+1},\texttt{ret\_exp}] =$ \texttt{expand($\eps_{k+1},\lambda_\mathrm{c},\ldots$)}	
		\COMMENT{where:}
		\STATE \COMMENT{\qquad $\real{\lambda_{k+1}} \ge \real{\lambda_\mathrm{c}}$}
		\STATE \COMMENT{\qquad \emph{\texttt{ret\_exp}} is the expansion's return code from Table~\ref{table:return_codes}}
		\STATE \texttt{expand\_converged $\coloneqq$ (\texttt{ret\_exp} $==1$ \OR \texttt{ret\_exp} $==3$)}	
		\IF{\texttt{expand\_converged} \AND $\real{\lambda_{k+1}} < 2 \cdot \tau_\eps$ }
			\RETURN \COMMENT{HEC converged to tolerance}
		\ELSIF {$\real{\lambda_{k+1} - \lambda_\mathrm{c}} > 0 $ }
			\STATE $\epslb \coloneqq 0$, $\epsub \coloneqq \eps_{k+1}$ \, \COMMENT{Expansion made some progress; do new contraction}
		\ELSIF {\texttt{ret\_con} $==3$}
			\RETURN \COMMENT{HEC stagnated}
		\ENDIF \, \COMMENT{Else contraction will be resumed/restarted from where it last left off} 
		
	\ENDFOR
\end{algorithmic}
\end{algorithm}
\end{algfloat}

We have designed the contraction procedure so that 
reaching its maximum allowed iteration count will  only cause it to actually halt iterating if it has also achieved some
amount of contraction, that is, it has encountered at least one point $\tilde\lambda$ such that $0 \le \real{\tilde\lambda} \le \lambda$,
where $\lambda$ is the initial point.  
This seemingly unconventional behavior has the two benefits that the only case when it doesn't make any progress is when it is 
impossible to do so (i.e. when it exhausts the machine's precision) and a sufficiently large maximum iteration limit to
find a first contraction step no longer needs to be known a priori, which is generally not possible.
If the contraction routine has made progress (no matter the termination condition), then additional 
expansion is always potentially possible and so the next expansion phase must be attempted.  Furthermore, even if the contraction phase 
failed to make any progress, but if the previous expansion did not converge, then the next expansion phase should 
attempt to \emph{resume} it since further expansion is apparently possible and may enable the subsequent contraction phase 
to finally make progress.  
The only case remaining is when the contraction phase failed to make any progress
(by reaching the limits of the hardware) \emph{after} having had the preceding expansion phase converge (meaning it would not be able make 
further progress if it were to be rerun with the same value of $\eps$).  In this situation, HEC can no longer make any progress 
and must quit.  However, even though the contraction phase failed to meet its convergence criteria, 
$\real{\lambda_c} \in [0,2\tau_\eps)$ may still hold, so HEC may sometimes terminate successfully in this case.  If not, 
HEC has stagnated, which is likely an indication that tolerances are too tight for the available precision of the hardware on the particular problem
(or possibly that a subroutine has failed in practice).  

For each expansion phase, we consider it to have converged once it can no longer make any meaningful rightward progress.  Our new
stopping criteria attempt to capture precisely that, and do so more accurately than the previous scheme.  
Furthermore, if the line search fails to produce a monotonic step, then the expansion routine is, by default, unable to make further
progress.  We have observed that the line search failing is generally a good sign, often implying that a stationary point has already 
been found.  We thus consider the expansion phase to have converged if either our new stopping condition is met or if the line
search fails.  Otherwise, further expansion is potentially possible.  After an expansion phase, HEC should first check 
if the expansion phase converged and whether $\real{\lambda} \in [0,2\tau_\eps)$ holds, as the two conditions together
indicate HEC has converged and can halt with success.  However, if the expansion phase has made progress, then, since it has
not converged,  HEC should continue by starting a new contraction 
phase.  Otherwise, we know that the expansion phase has not made any progress and is thus considered converged.  If 
the previous contraction phase exhausted the precision of the machine, then the HEC iteration can no longer continue and 
it has stagnated before meeting its convergence criteria for tolerances that are likely too tight.
The only remaining possibility is that the contraction phase achieved some amount of contraction but did not yet
converge.  In this last case, the contraction phase should be restarted from its most recent bracket to see if it can make further progress, which
might enable a subsequent expansion to succeed.  

The above design also causes the respective maximum iteration limits of the expansion and contraction phases to act as 
additional early termination features early within the HEC iteration, without ever comprising the final numerical accuracy.

\section{A new iteration for the complex stability radius}
\label{sec:complex_svsa}
Though we have developed Algorithm~SVSA-RF specifically to iterate over real-valued perturbations with rank at most two, it also permits a natural extension to a complex-valued rank-1 expansion iteration as well,
essentially by replacing $\real{uv^*}$ with $uv^*$ in \eqref{difference_eqn_sol}.
In Tables~\ref{table:dense_cont_csr}-\ref{table:sparse_csr}, we compare the original HEC algorithm 
with an alternative variant which employs this ODE-based expansion iteration for approximating 
the complex stability radius.  Overall there doesn't seem to be a clear answer as to which  
version performs better, as they perform roughly the same on many problems.  However, 
it is worth noting that ODE-based version had outstanding performance on continuous-time problems 
\texttt{ROC1} and \texttt{ROC2}, respectively requiring only 3.3\% and 16.6\% of the eigensolves 
as compared to the unmodified HEC algorithm.  Furthermore, on the latter example, the ODE-based
version also returns a significantly better approximation.

\begin{table}[h]
\small
\centering
\begin{tabular}{ l | cc | rr | rr | ll } 
\toprule 
\multicolumn{9}{c}{Small Dense Problems: Continuous-time}\\
\midrule
\multicolumn{1}{c}{} & \multicolumn{2}{c}{Iters} & \multicolumn{2}{c}{\# Eig} & \multicolumn{2}{c}{Time (secs)} & \multicolumn{2}{c}{CSR Approximation} \\
\cmidrule(lr){2-3}  
\cmidrule(lr){4-5}  
\cmidrule(lr){6-7}
\cmidrule(lr){8-9}
\multicolumn{1}{l}{Problem} & 
	\multicolumn{1}{c}{v1} & \multicolumn{1}{c}{v2} &	
	\multicolumn{1}{c}{v1} & \multicolumn{1}{c}{v2} &
	\multicolumn{1}{c}{v1} & \multicolumn{1}{c}{v2} &
	\multicolumn{1}{c}{$\min \{\eps_1,\eps_2\}$} & \multicolumn{1}{c}{$(\eps_1 - \eps_2)/\eps_1$} \\
\midrule

\texttt{CBM} & 4 & 4 & 169 & 242 & 29.136 & 39.991 & $3.80262077280 \times 10^{0}$ & $-1.5 \times 10^{-11}$\\
\texttt{CSE2} & 6 & 6 & 1298 & 1038 & 5.344 & 4.327 & $4.91778669279 \times 10^{1}$ & $+3.0 \times 10^{-5}$\\
\texttt{CM1} & \multicolumn{1}{c}{-} & \multicolumn{1}{c|}{-}  & \multicolumn{1}{c}{-} & \multicolumn{1}{c|}{-} & \multicolumn{1}{c}{-} & \multicolumn{1}{c|}{-} & \multicolumn{1}{c}{-} & \multicolumn{1}{c}{-}\\
\texttt{CM3} & 4 & 3 & 224 & 173 & 3.602 & 2.794 & $1.21736892616 \times 10^{0}$ & $-9.8 \times 10^{-3}$\\
\texttt{CM4} & 4 & 4 & 379 & 500 & 27.367 & 36.763 & $6.29458676972 \times 10^{-1}$ & $+1.2 \times 10^{-12}$\\
\texttt{HE6} & 11 & 11 & 20459 & 20981 & 12.524 & 12.637 & $2.02865555308 \times 10^{-3}$ & $-6.1 \times 10^{-11}$\\
\texttt{HE7} & 4 & 4 & 610 & 478 & 0.462 & 0.389 & $2.88575420548 \times 10^{-3}$ & \multicolumn{1}{c}{-}\\
\texttt{ROC1} & 3 & 3 & 4139 & 136 & 2.521 & 0.204 & $8.21970266187 \times 10^{-1}$ & \multicolumn{1}{c}{-}\\
\texttt{ROC2} & 5 & 3 & 481 & 80 & 0.361 & 0.131 & $7.49812117958 \times 10^{0}$ & $+9.1 \times 10^{-2}$\\
\texttt{ROC3} & 6 & 4 & 270 & 193 & 0.253 & 0.197 & $5.80347782972 \times 10^{-5}$ & $-6.2 \times 10^{-11}$\\
\texttt{ROC4} & 3 & 1 & 223 & 111 & 0.204 & 0.133 & $3.38236009391 \times 10^{-3}$ & $-2.7 \times 10^{-2}$\\
\texttt{ROC5} & 10 & 15.5 & 546 & 669 & 0.356 & 0.422 & $1.02041169816 \times 10^{2}$ & $+2.5 \times 10^{-7}$\\
\texttt{ROC6} & 4 & 4 & 183 & 127 & 0.176 & 0.145 & $3.88148973329 \times 10^{-2}$ & \multicolumn{1}{c}{-}\\
\texttt{ROC7} & 5 & 5 & 4439 & 4541 & 2.183 & 2.323 & $8.91295691482 \times 10^{-1}$ & \multicolumn{1}{c}{-}\\
\texttt{ROC8} & 3 & 3 & 236 & 124 & 0.226 & 0.167 & $1.51539044957 \times 10^{-1}$ & $-9.9 \times 10^{-9}$\\
\texttt{ROC9} & 6 & 6 & 705 & 674 & 0.509 & 0.479 & $3.03578083291 \times 10^{-1}$ & \multicolumn{1}{c}{-}\\
\texttt{ROC10} & 3 & 3 & 412 & 413 & 0.282 & 0.305 & $9.85411638072 \times 10^{0}$ & \multicolumn{1}{c}{-}\\

\bottomrule
\end{tabular}
\caption{
The columns are the same as described in the caption of Table~\ref{table:dense_cont} except that here, we compare
the HEC algorithm for approximating the complex stability radius (CSR) using its original expansion iteration, which
we call ``v1", and an ODE-based variant, which we call ``v2".  We tested both versions without any acceleration features enabled.  
Note that the ODE-based approach failed to find an upper bound for \texttt{CM1} and thus 
no data is reported for this example.
}
\label{table:dense_cont_csr}
\end{table}

\begin{table}
\small
\centering
\begin{tabular}{ l | cc | rr | rr | ll } 
\toprule 
\multicolumn{9}{c}{Small Dense Problems: Discrete-time}\\
\midrule
\multicolumn{1}{c}{} & \multicolumn{2}{c}{Iters} & \multicolumn{2}{c}{\# Eig} & \multicolumn{2}{c}{Time (secs)} & \multicolumn{2}{c}{CSR Approximation} \\
\cmidrule(lr){2-3}  
\cmidrule(lr){4-5}  
\cmidrule(lr){6-7}
\cmidrule(lr){8-9}
\multicolumn{1}{l}{Problem} & 
	\multicolumn{1}{c}{v1} & \multicolumn{1}{c}{v2} &	
	\multicolumn{1}{c}{v1} & \multicolumn{1}{c}{v2} &
	\multicolumn{1}{c}{v1} & \multicolumn{1}{c}{v2} &
	\multicolumn{1}{c}{$\min \{\eps_1,\eps_2\}$} & \multicolumn{1}{c}{$(\eps_1 - \eps_2)/\eps_1$} \\
\midrule

\texttt{AC5} & 3 & 3 & 177 & 173 & 0.165 & 0.163 & $1.31122274593 \times 10^{-2}$ & \multicolumn{1}{c}{-}\\
\texttt{AC12} & 3 & 3 & 1258 & 1290 & 0.789 & 0.768 & $9.24428279719 \times 10^{-2}$ & \multicolumn{1}{c}{-}\\
\texttt{AC15} & 5 & 5 & 143 & 143 & 0.156 & 0.142 & $4.22159665084 \times 10^{-2}$ & \multicolumn{1}{c}{-}\\
\texttt{AC16} & 4 & 4 & 940 & 931 & 0.553 & 0.537 & $5.46839166119 \times 10^{-2}$ & \multicolumn{1}{c}{-}\\
\texttt{AC17} & 4 & 7 & 237 & 313 & 0.198 & 0.268 & $3.33193351568 \times 10^{-6}$ & $+5.5 \times 10^{-10}$\\
\texttt{REA1} & 3 & 3 & 766 & 772 & 0.488 & 0.546 & $1.34439972514 \times 10^{-3}$ & \multicolumn{1}{c}{-}\\
\texttt{AC1} & 4 & 4 & 4787 & 4771 & 2.471 & 2.591 & $6.65190471447 \times 10^{0}$ & \multicolumn{1}{c}{-}\\
\texttt{AC2} & 4 & 4 & 511 & 570 & 0.357 & 0.393 & $3.27216087222 \times 10^{0}$ & \multicolumn{1}{c}{-}\\
\texttt{AC3} & 4 & 4 & 938 & 917 & 0.541 & 0.548 & $5.25195687767 \times 10^{-2}$ & \multicolumn{1}{c}{-}\\
\texttt{AC6} & 11 & 13 & 875 & 911 & 0.552 & 0.621 & $1.88905032255 \times 10^{-8}$ & $-9.6 \times 10^{-9}$\\
\texttt{AC11} & 7.5 & 6.5 & 286 & 347 & 0.245 & 0.268 & $4.57670218088 \times 10^{-8}$ & $+1.3 \times 10^{-8}$\\
\texttt{ROC3} & 4 & 4 & 1510 & 1485 & 1.022 & 1.027 & $4.27872787193 \times 10^{-2}$ & \multicolumn{1}{c}{-}\\
\texttt{ROC5} & 6 & 8 & 358 & 432 & 0.292 & 0.329 & $2.55709313478 \times 10^{-4}$ & $+8.0 \times 10^{-11}$\\
\texttt{ROC6} & 11 & 11 & 16832 & 16646 & 9.600 & 9.137 & $5.81391331473 \times 10^{-2}$ & \multicolumn{1}{c}{-}\\
\texttt{ROC7} & 4 & 4 & 3509 & 3480 & 1.855 & 1.928 & $9.01354009455 \times 10^{-1}$ & \multicolumn{1}{c}{-}\\
\texttt{ROC8} & 3 & 3 & 127 & 124 & 0.145 & 0.181 & $1.59160474017 \times 10^{-5}$ & $-1.2 \times 10^{-10}$\\
\texttt{ROC9} & 4 & 4 & 286 & 272 & 0.272 & 0.249 & $3.49507190967 \times 10^{-2}$ & \multicolumn{1}{c}{-}\\

\bottomrule
\end{tabular}
\caption{
See caption of Table~\ref{table:dense_cont_csr} for the description of the columns.
}
\label{table:dense_disc_csr}
\end{table}

\begin{table}
\small
\centering
\begin{tabular}{ l | cc | rr | rr | ll } 
\toprule 
\multicolumn{9}{c}{Large Sparse Problems: Continuous-time (top), Discrete-time (bottom)}\\
\midrule
\multicolumn{1}{c}{} & \multicolumn{2}{c}{Iters} & \multicolumn{2}{c}{\# Eig} & \multicolumn{2}{c}{Time (secs)} & \multicolumn{2}{c}{CSR Approximation} \\
\cmidrule(lr){2-3}  
\cmidrule(lr){4-5}  
\cmidrule(lr){6-7}
\cmidrule(lr){8-9}
\multicolumn{1}{l}{Problem} & 
	\multicolumn{1}{c}{v1} & \multicolumn{1}{c}{v2} &	
	\multicolumn{1}{c}{v1} & \multicolumn{1}{c}{v2} &
	\multicolumn{1}{c}{v1} & \multicolumn{1}{c}{v2} &
	\multicolumn{1}{c}{$\min \{\eps_1,\eps_2\}$} & \multicolumn{1}{c}{$(\eps_1 - \eps_2)/\eps_1$} \\
\midrule

\texttt{NN18} & 3 & 3 & 89 & 107 & 8.315 & 8.295 & $9.77172733234 \times 10^{-1}$ & \multicolumn{1}{c}{-}\\
\texttt{dwave} & 2 & 2 & 79 & 70 & 36.342 & 33.786 & $2.63019715625 \times 10^{-5}$ & \multicolumn{1}{c}{-}\\
\texttt{markov} & 2 & 2 & 57 & 65 & 12.179 & 13.279 & $1.61146532880 \times 10^{-4}$ & \multicolumn{1}{c}{-}\\
\texttt{pde} & 4 & 3 & 251 & 202 & 23.441 & 18.918 & $2.71186478815 \times 10^{-3}$ & $-4.4 \times 10^{-9}$\\
\texttt{rdbrusselator} & 4 & 4 & 289 & 369 & 91.167 & 112.681 & $5.35246333569 \times 10^{-4}$ & \multicolumn{1}{c}{-}\\
\texttt{skewlap3d} & 2 & 2 & 118 & 91 & 136.307 & 141.747 & $4.59992022215 \times 10^{-3}$ & \multicolumn{1}{c}{-}\\
\texttt{sparserandom} & 2 & 2 & 74 & 65 & 3.106 & 2.618 & $7.04698184526 \times 10^{-6}$ & \multicolumn{1}{c}{-}\\

\midrule

\texttt{dwave} & 2 & 2 & 42 & 33 & 21.303 & 14.683 & $2.56235064981 \times 10^{-5}$ & \multicolumn{1}{c}{-}\\
\texttt{markov} & 3 & 3 & 65 & 75 & 14.383 & 17.617 & $2.43146945130 \times 10^{-4}$ & \multicolumn{1}{c}{-}\\
\texttt{pde} & 5 & 3 & 2839 & 2091 & 184.350 & 133.007 & $2.68649594344 \times 10^{-4}$ & $-4.6 \times 10^{-4}$\\
\texttt{rdbrusselator} & 3 & 3 & 51 & 69 & 5.081 & 6.892 & $2.56948942080 \times 10^{-4}$ & \multicolumn{1}{c}{-}\\
\texttt{skewlap3d} & 2 & 2 & 47 & 56 & 58.705 & 49.745 & $3.40623440406 \times 10^{-5}$ & \multicolumn{1}{c}{-}\\
\texttt{sparserandom} & 2 & 2 & 24 & 22 & 0.927 & 0.897 & $2.53298721307 \times 10^{-7}$ & $-8.2 \times 10^{-10}$\\
\texttt{tolosa} & 4 & 4 & 2314 & 2315 & 598.283 & 547.306 & $1.76587075884 \times 10^{-7}$ & \multicolumn{1}{c}{-}\\

\bottomrule
\end{tabular}
\caption{
See caption of Table~\ref{table:dense_cont_csr} for the description of the columns.
}
\label{table:sparse_csr}
\end{table}

\end{appendices}

\end{document}